\documentclass[a4paper,12pt]{article}
\usepackage{a4wide}
\usepackage{amssymb, amsmath}
\usepackage[amsmath,hyperref,thmmarks]{ntheorem}
\usepackage{t1enc,times}
\usepackage[utf8]{inputenc}
\usepackage[english]{babel}
\usepackage{enumerate,xfrac}
\usepackage{hyperref}
\usepackage{color,caption,cancel}
\usepackage[numbers]{natbib}
\usepackage{graphicx, subfig}
\graphicspath{{fig}{}}
\emergencystretch=\hsize
\def\cl{{\mathcal L}}

\def\cn{{\mathcal N}}

\def\E{{\mathbb E}}

\def\N{{\mathbb N}}
\def\P{{\mathbb P}}

\def\R{{\mathbb R}}

\def\t{\theta}

\def\s{\star}
\def\noH{\cancel{H}}

\def\abs#1{\left|#1\right|}
\def\tr{\mathop{\rm tr}\nolimits}

\def\inv#1{\mathop{\frac{1}{ #1}}\nolimits}
\def\expp#1{\mathop {\mathrm{e}^{ #1}}}

\theoremstyle{plain}
\newtheorem{theorem}{Theorem}[section]
\newtheorem{proposition}[theorem]{Proposition}
\newtheorem{lemma}[theorem]{Lemma}
\newtheorem{corollary}[theorem]{Corollary}
\newtheorem{remarkh}[theorem]{Remark}
\newenvironment{remark}{\begin{remarkh}\rm}{\end{remarkh}}
\newtheorem{definition}[theorem]{Definition}
\theoremstyle{nonumberplain}
\theoremheaderfont{\normalfont\itshape}
\theorembodyfont{\normalfont}
\theoremseparator{.}
\theoremsymbol{\ensuremath{\blacksquare}}
\newtheorem{proof}{Proof}

\makeatletter
\newcounter{hypo}
\renewcommand{\thehypo}{(${\mathcal H}$-\arabic{hypo})}
\newcommand{\dohypo}{\thehypo}
\newenvironment{hypo}[1][]{%
\refstepcounter{hypo}
\list{}{%
\settowidth{\labelwidth}{\dohypo}%
\setlength{\labelsep}{10pt}%
\setlength{\leftmargin}{\labelwidth}
\advance\leftmargin\labelsep%
}%
\item[\dohypo  #1]%
  }{%
\endlist
}

\makeatother

\title{How many inner simulations to compute conditional expectations with least-square Monte Carlo?}

\date{\today}
\author{Aurélien Alfonsi\thanks{CERMICS, Ecole des Ponts, Marne-la-Vall\'ee, France. MathRisk, Inria, Paris,
  France.\newline
\indent email: \texttt{aurelien.alfonsi@enpc.fr}} \and Bernard Lapeyre\thanks{CERMICS, Ecole des Ponts, Marne-la-Vall\'ee, France. MathRisk, Inria, Paris,
  France.\newline
\indent email: \texttt{bernard.lapeyre@enpc.fr}} \and Jérôme Lelong\thanks{Univ. Grenoble Alpes, CNRS, Grenoble INP, LJK, 38000 Grenoble, France. \newline
\indent email: \texttt{jerome.lelong@univ-grenoble-alpes.fr}}}

\begin{document}
\maketitle
\begin{abstract}The problem of computing the conditional expectation $\E[f(Y)|X]$ with least-square Monte-Carlo is of general importance and has been widely studied. To solve this problem, it is usually assumed that one has as many samples of~$Y$ as of~$X$. However, when samples are generated by computer simulation and the conditional law of $Y$ given~$X$ can be simulated, it may be relevant to sample~$K\in \N$ values of~$Y$ for each sample of~$X$. The present work determines the optimal value of~$K$ for a given computational budget, as well as a way to estimate it. The main take away message is that the computational gain can be  all the more important as the computational cost of sampling $Y$ given $X$ is small with respect to the computational cost of sampling~$X$. Numerical illustrations on the optimal choice of~$K$ and on the computational gain are given on different examples including one inspired by risk management.
\end{abstract}

{\bf AMS 2020:} 65C05, 91G60\\
{\bf Acknowledgement:} AA and BL acknowledge the support of the chaire ``Risques financiers'', Fondation du Risque.\\
{\bf Declarations:}\\
$\bullet$ {\bf Competing interests:} The authors have no competing interests as defined by Springer, or other interests that might be perceived to influence the results and/or discussion reported in this paper.\\
$\bullet$ {\bf Data availability:} Not applicable.

\section{Introduction and Framework}

We consider the classical problem of computing a conditional expectation using a least-square Monte Carlo approach. To be more precise, let $(\Omega,\mathcal{F},\mathbb{P})$ be a probability space, $X:\Omega \to \R^d$ and $Y: \Omega \to \R^p$ be two random variables  and $f : \R^p \to \R$ be a measurable function such that $\E[{f(Y)}^2] < \infty$. We are interested in computing $\E[f(Y) | X]$ by using a parametrized approximation. Thus, we introduce a family of measurable functions $(\varphi(\t, \cdot))_{\t \in \R^q}$ from $\R^d$ to $\R$ satisfying for all $\t \in \R^q$, $\E[{\varphi(\t, X)}^2] < \infty$. This family will be used to approximate the conditional expectation $\E[f(Y) | X]$. It is well known that $\E[f(Y) | X]$ solves the two following minimisation problems
\begin{align*}
  \inf_{Z \in L^2(\Omega,\sigma(X))} \E[(Z - f(Y))^2], \ &\inf_{Z \in L^2(\Omega,\sigma(X))} \E[(Z - \E[f(Y) |X])^2],
\end{align*}
where $L^2(\Omega,\sigma(X))$ denotes the set of square integrable random variables that are measurable with respect to the $\sigma$-algebra generated by~$X$. Therefore, we are interested in the following minimization problems
\begin{align}\label{minim_pbms}
  \inf_{\t  \in \R^q} \E[(\varphi(\t,X) - f(Y))^2], \ &\inf_{\t \in \R^q} \E[(\varphi(\t,X) - \E[f(Y) |X])^2].
\end{align}
In practical cases, all these expectations are not explicit and it is often used Monte-Carlo estimators to approximate them. The classical problem of regression consists in minimizing $\inv{N} \sum_{i=1}^N \left( \varphi(\t, X_i) - f(Y_i) \right)^2$ with respect to~$\t$, where $(X_i,Y_i)_{i\ge 1}$ is a sequence of iid random variables with the same distribution as $(X,Y)$. In this work, we consider the possibility of having for each $X_i$ many samples of $Y$ given $X_i$. This is the case when samples are generated by computer simulation and when the conditional law of $Y$ given~$X$ can be simulated. More precisely, let $(X_i)_{i \ge 1} $ be a sequence of iid random variables following the distribution of $X$. For each $i \ge 1$, we introduce independent sequences $(Y_i^{(k)})_{k \ge 1}$ of iid random variables following the law $\cl(Y | X = X_i)$ of $Y$ conditionally on $X=X_i$. For $N,K\in \N^*$, we define the sequence of functions $v_N^K: \R^q \to \R$ by
\begin{align}
  \label{eq:vKN}
  v_N^K(\t) = \inv{N} \sum_{i=1}^N \left( \varphi(\t, X_i) - \inv{K} \sum_{k=1}^K f(Y_i^{(k)}) \right)^2.
\end{align}
We are interested in finding $\t_N^K$ minimizing $v_N^K$, so that $\varphi(\t_N^K, X)$ will give an approximation of $\E[f(Y)|X]$. Formally, the two minimisation problems of Equation~\eqref{minim_pbms} correspond respectively to $N=\infty, K=1$ and $N=K=\infty$. Note that the minimisation of~\eqref{eq:vKN} with $K=1$ corresponds to the classical case, with as many samples of $Y_i$ as of $X_i$. Up to our knowledge, most of the literature (if not all) considers the case of minimizing $v_N^1$ to approximate the conditional expectation, and we refer to Gyorfi et al.~\cite{GKKW} for a nice presentation of the topic and references. This may be understood from the point of view of statistics: on empirical data, one usually have as many observations of $X$'s and $Y$'s. However, when $X$ and $Y$ are generated by computer simulation, it is relevant to consider the possibility of sampling  $K\ge 2$ values of $Y$ for a given~$X$. The natural question then is to understand how to choose $N$ and $K$ in order to achieve the best accuracy for a given computational time. This is the goal of this paper.

The problem of computing conditional expectations is an important problem that arises in many different fields of research, such as the approximation of backward stochastic differential equations~\cite{BT,GLW},  the pricing of American options and more generally optimal stopping problems~\cite{LS}, and stochastic optimal control problems~\cite{BKS} to mention a few. It has a particular relevance in risk management, see e.g.~\cite{BaResi2,BDM,KrNiKo}, where financial institutions have to evaluate risk from a regulatory perspective. The valuation of future risks naturally involves conditional expectations. To be more precise, let us consider the case of insurance companies that have to calculate their Solvency Capital Requirement (SCR). This SCR can be calculated by computing expected losses under some stressed scenarios. This regulatory procedure to evaluate risk is called the ``standard formula''.  If one aims at evaluating the SCR at a future date~$T$ with the same procedure, one has to compute conditional expected losses under the different stressed scenarios, given all the market information between the current date and~$T$, see~\cite{alfonsi2021multilevel}. Therefore, one naturally has to deal with the numerical approximation of  conditional expectations. Let us stress that it is usually natural in this context  to be able to sample conditional laws: assets are usually modeled by a Markovian process that can be simulated, and we can then simulate as many paths as desired after~$T$, from a given path up to time~$T$. 

Many works have developed numerical methods based on nested
simulations and refinements to approximate an expectation that
involves conditional expectations. Among them we can cite \cite{GoJu}
which optimize nested simulations to estimate a value at risk on a
conditional expectation, \cite{BaReSi} which study nested simulations
in the context of risk insurance modeling and
\cite{alfonsi2021multilevel} which use a multilevel approach on the
same kind of insurance problem. But to the best of our knowledge, none
of these works are interested in the nested approximation of
conditional expectations using a parametric representation as done in
this work.

The paper is structured as follows. First, Section~\ref{Sec_AC} presents our main assumptions, under which we are able to show, by quite standard arguments, the convergence of~$\t^K_N$ as well as a Central Limit Theorem. Section~\ref{Sec_Main} presents the main results of the paper. In particular, Theorem~\ref{thm:optimK} gives a precise asymptotic of the suboptimality of~$\t^K_N$ (with respect to $\t^\s$) as a function of $K$ and $N$, and the optimal value of $K$  for a given computational budget. It also gives estimators to approximate it. The computational gain is all the more important as the computational cost of sampling $Y$ given $X$ is small and the approximation family is close to the conditional expectation. Section~\ref{Sec_linearcase} gives a focus on the particularly important case of linear regression. We are able to refine the result of Theorem~\ref{thm:optimK} in this context. Besides, Proposition~\ref{prop:linear_reg} shows for a particular choice of approximating functions that the optimal value of~$K$ can be arbitrarily large.  Finally, Section~\ref{Sec_num} presents numerical results and shows the relevance of considering $K>1$ on different examples. We also compare different estimators that approximate~$K$, and it comes out that one estimator is more relevant for practical use.

\section{Assumptions and Convergence results}\label{Sec_AC}

In this section, we apply the general results on the convergence of the estimators of the optimal solutions presented by~\cite[Section 2.6]{ShRu93}. We introduce the function
\begin{align}
   v^\infty(\t) & = \E\left[ \left( \varphi(\t, X) - \E[f(Y) | X] \right)^2 \right] \label{def_v_infty}
\end{align}
and make the following assumptions.

\paragraph{Assumptions} Let $C\subset \R^q$ be a compact set with a non-empty interior~$\mathring{C}$.
\begin{hypo}
  \label{hyp:ui}
  Uniform integrability: $\E\left[\sup_{\t \in C} \abs{\varphi(\t, X)}^2\right] < \infty$.
\end{hypo}
\begin{hypo}
  \label{hyp:continuous}
  The function $\t \longmapsto \varphi(\t, X)$ is a.s. continuous on~$C$.
\end{hypo}
\begin{hypo}
  \label{hyp:unique}
  The function $v^\infty$ admits on~$C$ a unique minimizer $\t^\s \in \mathring{C}$. 
\end{hypo}
\begin{hypo}
  \label{hyp:C2-ui}
  The function $\t \longmapsto \varphi(\t, X)$ is a.s. twice continuously differentiable on~$C$ and such that
  \begin{equation*}
    \E\left[\sup_{\t \in C} \abs{\nabla \varphi(\t, X)}^2 \right] < \infty; \quad
    \E\left[\sup_{\t \in C} \abs{\nabla^2 \varphi(\t, X)}^2 \right] < \infty.
  \end{equation*}
\end{hypo}
Here, and in the whole paper, the gradient $\nabla$ is taken with respect to~$\theta$. Let us note that  Hypotheses~\ref{hyp:ui},~\ref{hyp:continuous} and~\ref{hyp:C2-ui} are satisfied in the case of the linear regression, see Section~\ref{Sec_linearcase} for further details.

To apply the results on the convergence of the estimators presented by~\cite[Section 2.6]{ShRu93}, we  introduce the function $$\Phi(\theta,Z)=\left(\varphi(\theta,X)-\frac 1K \sum_{k=1}^Kf( Y^{(k)})\right)^2,$$ with $Z=(X,Y^{(1)},\dots,Y^{(K)})$, where the sequence $(Y^{(k)})_{k\ge 1}$ is conditionally iid given~$X$, and given $X=x$ follows the distribution $\cl(Y | X=x)$\footnote{In practice, we typically have  $Y=F(X,U)$ with $U$ independent of~$X$ and $F$ a measurable function, and the conditional independence means that $Y^{(k)}=F(X,U^{(k)})$ with $X,U^{(1)},\dots,U^{(K)}$ independent and $\mathcal{L}(U^{(i)})=\mathcal{L}(U)$. }. We also define, for $K\in \N^*$,  the function
\begin{align}\label{def_v_K}
  v^K(\t) = \E\left[ \left( \varphi(\t, X) - \inv{K} \sum_{k=1}^K f(Y^{(k)}) \right)^2 \right],
\end{align}
and $v^K_N(\theta)=\frac 1 N \sum_{i=1}^N\Phi(\theta,Z_i)$, so that $v^K(\theta)=\E[v^K_N(\theta)]$.
Since $|\Phi(\theta,Z)|\le  2|\varphi(\theta,X)|^2+\frac 2K \sum_{k=1}^K f(Y^{(k)})^2$, we get the uniform integrability on~$C$ by using~\ref{hyp:ui}. The continuity of $\Phi$ with respect to~$\theta\in C$ is clear by~\ref{hyp:continuous}, and we get the following lemma from the uniform law of large numbers, see Lemma~\ref{lem:ulln}.
\begin{lemma}
  Under~\ref{hyp:ui} and~\ref{hyp:continuous}, for every fixed $K \in \N$,
  $\sup_{\theta \in C}|v_N^K(\theta)-v^K(\theta)| \to 0$ almost surely as $N\to \infty$.
\end{lemma}
The next lemma makes explicit the link between~$v^\infty(\t)$ and~$v^K(\t)$ defined respectively by~\eqref{def_v_infty} and~\eqref{def_v_K}.
\begin{lemma}\label{lem_vk_vinfty}
  We have for all $\t \in C$,
  \begin{align}
  v^K(\t) & = v^\infty(\t) + \inv{K} \E[(f(Y)- \E[f(Y) | X])^2 ]. \label{expression_vK}
\end{align}
\end{lemma}
\begin{proof}
  We expand~\eqref{def_v_K} and get
\begin{align*}
  v^K(\t) = \inv{K} \E\left[ \left( \varphi(\t, X) - f(Y) \right)^2 \right] + \frac{1}{K^2} \sum_{k \neq k'}\E\left[ \left( \varphi(\t, X) - f(Y^{(k)}) \right) \left( \varphi(\t, X) - f(Y^{(k')}) \right) \right].
\end{align*}
On the one hand, using the conditional independence of the $Y^{(k)}$'s, we get for $k\not = k'$
\begin{align*}
  & \E\left[ \left( \varphi(\t, X) - f(Y^{(k)}) \right) \left( \varphi(\t, X) - f(Y^{(k')}) \right) \right] \\
  & = \E\left[ \E\left[ \varphi(\t, X) - f(Y^{(k)}) | X \right] \E\left[ \varphi(\t, X) - f(Y^{(k')}) | X \right] \right] \\
  & = \E\left[ \left( \varphi(\t, X) - \E[f(Y) | X] \right)^2  \right].
\end{align*}
On the other hand, as the conditional expectation is an orthogonal projection,
\begin{align*}
  \E\left[ \left( \varphi(\t, X) - f(Y) \right)^2 \right] & = \E\left[ \left( (\varphi(\t, X) - \E[f(Y) | X]) + (\E[f(Y) | X] - f(Y)) \right)^2 \right] \\
  & = \E\left[ \left( \varphi(\t, X) - \E[f(Y) | X]\right)^2\right] + \E\left[\left(\E[f(Y) | X] - f(Y) \right)^2 \right].
\end{align*}
This yields to the claim.
\end{proof}

Let $\t_N^K$ (resp. $\t^\s$) be a minimizer of $v_N^K$ (resp. $v^\infty$) on the compact set~$C$, i.e.
\begin{align*}
  v_N^K(\t_N^K) = \inf_{\t \in C} v_N^K(\t) \quad \text{and} \quad v^\infty(\t^\s) = \inf_{\t \in C} v^\infty(\t).
\end{align*}
By Lemma~\ref{lem_vk_vinfty}, $v^K$ and $v^\infty$ differ only by a constant. So, $\t^\s$ is also the unique minimizer of $v^K$ for every $K$. Therefore, we have the following result from~\cite[Theorem A1, p.~67]{ShRu93}.
\begin{proposition}
  \label{prop:slln}
  Under~\ref{hyp:ui},~\ref{hyp:continuous},~\ref{hyp:unique}, for every fixed $K$, $\t_N^K \to \t^\s$ a.s. when $N \to \infty$.
\end{proposition}
Beside this almost sure convergence result, we also have a central limit theorem under additional assumptions.
\begin{proposition}
  \label{prop:tcl}
  Under the assumptions of Proposition~\ref{prop:slln},~\ref{hyp:C2-ui} and if $\E\left[\left( \varphi(\t^\s, X) - \inv{K} \sum_{k=1}^K f(Y^{(k)}) \right)^2 | \nabla \varphi(\t^\s, X)|^2 \right]<\infty$ and the matrix $H:=\nabla^2 v^\infty(\t^\s)$ is positive definite, we have
  \begin{align}\label{conv_thetaKN}
    \sqrt{N} (\t_N^K - \t^\s) \xrightarrow[N \to \infty]{\cl} \cn(0, 4 H^{-1} \Gamma^{K} H^{-1})
  \end{align}
  with \begin{equation}\label{Gamma_K_expr}
    \Gamma^K=A+B/K,
  \end{equation}
where $A,B\in \R^{q\times q}$ are the following semi-definite positive matrices:
  \begin{align}
  A&=\E\left[ (\varphi(\t^\s, X) - \E[f(Y) | X])^2 \nabla \varphi(\t^\s, X) \nabla \varphi(\t^\s, X)^T  \right], \label{def_A}\\
  B&=\E\left[ \left(f(Y)- \E[f(Y) | X] \right)^2 \nabla \varphi(\t^\s, X) \nabla \varphi(\t^\s, X)^T \right]. \label{def_B}
\end{align}
  Furthermore, we have \begin{equation}\label{asymp_vinf}
  N(v^\infty(\t^K_N)-v^\infty(\t^\star))\xrightarrow[N \to \infty]{\cl}
  2 G^THG \text{ with } G\sim\cn(0, H^{-1} \Gamma^{K}
  H^{-1}).
\end{equation}
\end{proposition}
\begin{proof}
  First, we check some properties on gradients. We have $\nabla \Phi(\theta,Z)=2(\varphi(\theta,X)-\frac 1K \sum_{k=1}^Kf( Y^{(k)})) \nabla\varphi(\theta,X)$ and $\nabla^2\Phi(\theta,Z)= 2\nabla\varphi(\theta,X)\nabla\varphi(\theta,X)^T+2(\varphi(\theta,X)-\frac 1K \sum_{k=1}^K f(Y^{(k)})) \nabla^2\varphi(\theta,X)$. From Cauchy-Schwarz inequality,~\ref{hyp:C2-ui} and $\E[f(Y)^2]<\infty$, we get that $\sup_{\theta \in C}|\nabla \Phi(\theta,Z)|$ and  $\sup_{\theta \in C}|\nabla^2 \Phi(\theta,Z)|$ are integrable. Besides, the matrix
  \begin{align}
    \label{eq:Gamma_K}
    \Gamma^{K} = \E\left[ \left( \varphi(\t^\s, X) - \inv{K} \sum_{k=1}^K f(Y^{(k)}) \right)^2 \nabla \varphi(\t^\s, X)  \nabla \varphi(\t^\s, X)^T  \right].
  \end{align}
  is well defined since $\E\left[\left( \varphi(\t^\s, X) - \inv{K} \sum_{k=1}^K f(Y^{(k)}) \right)^2 | \nabla \varphi(\t^\s, X)|^2 \right]<\infty$, and we get~\eqref{conv_thetaKN} following the result from~\cite[Theorem A2, p.~74]{ShRu93}.

  Let us check that $\Gamma^{K} =A+B/K$. We have
\begin{align*}
  \Gamma^{K} & = \frac{1}{K^2} \E\left[ \E\left[\left( \sum_{k=1}^K \varphi(\t^\s, X) -  f(Y^{(k)}) \right)^2 | X \right] \nabla \varphi(\t^\s, X)  \nabla \varphi(\t^\s, X)^T  \right] \\
  & = \frac{1}{K} \E\left[ \E\left[(\varphi(\t^\s, X) -  f(Y))^2 | X \right] \nabla \varphi(\t^\s, X) \nabla \varphi(\t^\s, X)^T \right] \\
  & \quad + \frac{K-1}{K} \E\left[ \left(\varphi(\t^\s, X) - \E[f(Y) | X] \right)^2 \nabla \varphi(\t^\s, X) \nabla \varphi(\t^\s, X)^T \right] \\
  & =  \E\left[ (\varphi(\t^\s, X) - \E[f(Y) | X])^2 \nabla \varphi(\t^\s, X) \nabla \varphi(\t^\s, X)^T  \right]
  \\ & \quad + \frac{1}{K} \E\left[ \left(f(Y)- \E[f(Y) | X] \right)^2 \nabla \varphi(\t^\s, X) \nabla \varphi(\t^\s, X)^T \right]=A+\frac{B}K.
\end{align*}

Last, we have  $v^\infty(\t)-v^\infty(\t^\star)=\frac 12 (\t-\t^\star)^T H(\t-\t^\star) +|\t-\t^\star|^2\varepsilon(|\t-\t^\star|)$ with $\varepsilon(h)\to 0$ as $h\to 0$. By Slutsky's theorem, Proposition~\ref{prop:slln} and~\eqref{conv_thetaKN}, we get~\eqref{asymp_vinf}.
\end{proof}

\begin{remark} Note that the matrix~$A$ defined by~\eqref{def_A}
  corresponds to the asymptotic variance of the optimal regression that we would obtain if we
  could directly sample $\E[f(Y)|X]$. The additional term~$B/K$ in the
  decomposition of $\Gamma^{K}$ is the extra variance generated by
  the Monte Carlo approximation of $\E[f(Y)|X]$.
\end{remark}

Unless in very specific cases where the function $v^\infty$ is explicit, it is impossible in practice to numerically evaluate $N(v^\infty(\t^K_N)-v^\infty(\t^\s))$.  The next proposition shows that $N(v^K_N(\t^\s)-v^K_N(\t^K_N))$ has the same asymptotics as~\eqref{asymp_vinf}. Roughly speaking, the suboptimality of $\t^\s$ for $v^K_N$ is of the same order as the suboptimality of $\t^K_N$ for $v^\infty$. This result will be used in the numerical section~\ref{Sec_num} to illustrate the convergence. 

\begin{proposition}\label{prop:asymp} Under the same assumptions as in Proposition~\ref{prop:tcl} and if $C$ is convex, we have
  $$N(v^K_N(\t^\s)-v^K_N(\t^K_N))\xrightarrow[N \to \infty]{\cl}
  2 G^THG, \text{ with } G\sim\cn(0, H^{-1} \Gamma^{K}  H^{-1}). $$
\end{proposition}
\begin{proof}
  We have by Taylor's theorem
\begin{equation}\label{Taylor_hess} v^K_N(\t^\s)-v^K_N(\t^K_N)= (\t^K_N-\t^\s)^T  \left( \int_0^1(1-u) \nabla^2v^K_N\left(\t^K_N+u (\t^K_N-\t^\s) \right) du \right) (\t^K_N-\t^\s),
\end{equation}
  with
  $$ \nabla^2v^K_N(\theta)=\frac 2N \sum_{i=1}^N \nabla \varphi(\theta,X_i)\nabla \varphi(\theta,X_i)^T +\left(\varphi(\theta,X_i)-\frac 1K \sum_{k=1}^K f(Y_i^{(k)})\right) \nabla^2\varphi(\theta,X_i).$$
By Lemma~\ref{lem_vk_vinfty}, we have $$ \nabla^2v^K(\theta)= \nabla^2v^\infty(\theta)= 2\E\left[  \nabla \varphi(\theta,X)\nabla \varphi(\theta,X)^T +\left(\varphi(\theta,X)-\E[f(Y)|X]\right) \nabla^2\varphi(\theta,X)\right].$$
By~\ref{hyp:ui},~\ref{hyp:continuous} and~\ref{hyp:C2-ui}, we can apply~\cite[Lemma A1 p.~67]{ShRu93} and get that $\sup_{\t \in C} |\nabla^2v^K_N(\theta)- \nabla^2v^\infty(\theta)| \underset{N\to \infty}\to 0$, almost surely. Since $\t^\s,\t^K_N\in C$ and $C$ is convex, we get $ \int_0^1(1-u) \nabla^2v^K_N\left(\t^K_N+u (\t^K_N-\t^\s) \right) du - \int_0^1(1-u) \nabla^2v^\infty\left(\t^K_N+u (\t^K_N-\t^\s) \right) du \to 0$, almost surely. Since $\nabla^2v^\infty$ is bounded on~$C$, we get
$$\int_0^1(1-u) \nabla^2v^K_N\left(\t^K_N+u (\t^K_N-\t^\s) \right) du \underset{N\to \infty} \to H=\nabla^2v^K(\t^\s), a.s.,$$
by using Proposition~\ref{prop:slln}.  This  gives  that $N(v^K_N(\t^\s)-v^K_N(\t^K_N))$ converges in law to $2G^THG$ by using~\eqref{Taylor_hess}, Proposition~\ref{prop:tcl} and Slutsky's theorem.
\end{proof}

\section{Main results}\label{Sec_Main}

In this section, we present our main theorem that determines the optimal allocation between~$N$ and $K$ to approximate the conditional expectation. Let us denote the computational time for sampling $X$ and $\cl(Y | X)$  respectively by $C_X$ and $C_{Y | X}$. With these notations, the cost for computing $v_N^K$ is proportional to $N C_X + N K C_{Y | X}$. Without loss of generality, we will assume that $C_X=1$. This means that the computational time for sampling $X$ is one unit and that we express all the other computational times with respect to this unit.

We now discuss the computational cost of calculating $\theta^K_N$ by using a gradient descent type method. Let us observe from the definition of $v^K_N$ in Equation~\eqref{eq:vKN} that  $\left(\inv{K} \sum_{k=1}^K f(Y_i^{(k)})\right)_{1\le i\le N}$ can be computed  once and for all. Then, the gradient descent applied to~$v^K_N$ has exactly the same computational cost as the one applied to $v^1_N$. This is why we do not include the cost of calculating $\theta^K_N$ in our reasoning and only focus on the computational cost of $v^K_N$.

\subsection{Optimal allocation between \texorpdfstring{$N$}{N} and \texorpdfstring{$K$}{K}}

\begin{definition}\label{def_nu} For $x>0$, we denote by $\nu(x) \in \N^*$ the unique natural number such that $$(\nu(x)-1)\nu(x) <x\le \nu(x)(\nu(x)+1).$$
\end{definition}
It is easy to check that $\forall x>0,  \lfloor \sqrt{x} \rfloor  \le \nu(x)\le \lceil \sqrt{x} \rceil. $ Now, we state our main result.

\begin{theorem}\label{thm:optimK}
  Under the assumptions of Proposition~\ref{prop:tcl} and if the sequence $N(v^\infty(\t^K_N)- v^\infty(\t^\star))_{N\ge 1}$ is uniformly integrable, we have
  $$\E[v^\infty(\t^K_N)] = v^\infty(\t^\star)+\frac{\tr(\Gamma^{K}H^{-1})}N+o(1/N),$$
  as $N\to \infty$. If $A\not =0$, the asymptotic optimal choice minimizing $\E[v^\infty(\t^K_N)]$ for a computational budget $c\to \infty$ is to take
  $$N^\star= \left \lfloor \frac{c }{ 1 + K^\star C_{Y|X}} \right\rfloor, \  K^\star= \nu\left(\frac{\tr(BH^{-1})}{C_{Y|X}\tr(AH^{-1})} \right).$$
\end{theorem}
Note that if $A=0$ and $\P(\nabla \varphi(\t^\s,X) \not = 0)>0$, then $\varphi(\t^\s,X)=\E[f(Y)|X]$. The condition $A\not=0$ ensures that $\tr(AH^{-1})>0$.
\begin{proof}
 We have  by Proposition~\ref{prop:tcl}, $N(v^\infty(\t^K_N)-v^\infty(\t^\star))\xrightarrow[N \to \infty]{\cl}  2 G^THG$ with  $G\sim\cn(0, H^{-1} \Gamma^{K}  H^{-1})$.
From this convergence in distribution and the uniform integrability assumption, we get  $\E[N(v^\infty(\t^K_N)-v^\infty(\t^\star))]\to 2  \E[G^THG]$. Let $C=\sqrt{H^{-1} \Gamma^{K} H^{-1}}$. Then, $G$ has the same law as $C\tilde{G}$ with $\tilde{G}\sim \cn(0,I_q)$ and thus   $\E[G^THG]=\E[\tilde{G}^T CHC\tilde{G} ]=\tr(CHC)=\tr(HC^2)=\tr(\Gamma^{K} H^{-1})$, which gives the first claim.

As $N\to \infty$,  the minimization of $\E[v^\infty(\t^K_N)]$ with respect to~$K$ amounts to minimizing $\tr(\Gamma^{K}H^{-1})$ with respect to~$K$. For a large enough budget $c$, the problem becomes
  \begin{equation*}
    \inf_{\substack{
      N, K \, \in \, \N \\
      s.t. \, N  + N K C_{Y | X} = c} } \frac{\tr(AH^{-1})}{N} + \frac{\tr(BH^{-1})}{KN}.
  \end{equation*}
  Then, we apply Lemma~\ref{lem:optim} to get the claim.
\end{proof}

\begin{remark}\label{rk_KnoH}
  Theorem~\ref{thm:optimK} gives the asymptotic optimal allocation to minimize $\E[v^\infty(\t^K_N)]$. Unfortunately, it involves the matrix $H$ which is in general unknown and may be difficult to estimate. When $\theta$ is a one dimensional parameter, $A$, $B$ and $H$ are scalar values and thus $K^\star= \nu\left(\frac{B}{C_{Y|X}A} \right)$. Otherwise, since $H$ is a definite positive matrix, we have $\underline{\lambda}_H I_q\le H\le \overline{\lambda}_H I_q$ and thus
  $$ \overline{\lambda}_H^{-1}\tr(\Gamma^{K})\le \tr(\Gamma^{K}H^{-1})\le \underline{\lambda}_H^{-1}\tr(\Gamma^{K}).$$
Therefore, it is reasonable (though not optimal) to minimize  $\tr(\Gamma^{K})$ under the same computational budget constraint, which then leads to
 $$N'= \left \lfloor \frac{c }{ 1 + K' C_{Y|X}} \right\rfloor, \  K'= \nu\left(\frac{\tr(B)}{C_{Y|X}\tr(A)} \right).$$
\end{remark}

The next corollary gives a bound on the computational gain that can be obtained by the optimization of~$K$ given by Theorem~\ref{thm:optimK}.
\begin{corollary}[Comparison of the estimators $\t^{K^\s}_{N^\s}$ and $\t^1_N$ for a fixed computational budget] \label{cor:gain} Let $c$ be the computational budget. Under the assumptions of Theorem~\ref{thm:optimK} and with $N=\lfloor c/(1+C_{Y|X})\rfloor$, we have $$\E[v^\infty(\t^{K^\s}_{N^\s})]-v^\infty(\t^\s)\sim_{c\to \infty} r^\s(\E[v^\infty(\t^{1}_{N})]-v^\infty(\t^\s)),$$ with
 \begin{equation}\label{def_rs}r^\s=\frac{(1+\nu(\xi) C_{Y|X})\left(1+\frac{\xi}{\nu(\xi)} C_{Y|X}\right)}{(1+ C_{Y|X})(1+\xi C_{Y|X})}
  , \ \xi=\frac{\tr(BH^{-1})}{C_{Y|X} \tr(AH^{-1})}.
    \end{equation}
  This  multiplicative gain $r^\s \in (0,1]$  satisfies $r^\s\ge \frac{C_{Y|X}}{1+C_{Y|X}}$ and $\lim_{\xi \to \infty}r^\s = \frac{C_{Y|X}}{1+C_{Y|X}}$.
\end{corollary}
Note that if $C_{Y|X}=1$, i.e. the computation time of sampling $\mathcal{L}(Y|X)$ is the same as the one of sampling~$X$, we cannot reduce the computation time more than by a factor~$1/2$. Besides, the smaller is $C_{Y|X}$, the more we may hope a significant reduction of computational time, and this really occurs if $\xi$ is large, so that $r^\s \approx \frac{C_{Y|X}}{1+C_{Y|X}}$.
\begin{proof}
Since we are comparing $\t^1_N$ and $\t^{K^\s}_{N^\s}$ for the same computation budget, we have $c\sim_{c\to \infty} N(1+C_{Y|X})\sim_{c\to \infty} N^\s(1+K^\s C_{Y|X})$.
  By Theorem~\ref{thm:optimK}, the multiplicative gain in precision is
  $$r^\s=\frac{\tr(\Gamma^{K^\s}H^{-1})/N^\s}{\tr(\Gamma^{1}H^{-1})/N} \underset{c \to \infty}\to\frac{\tr(\Gamma^{K^\s}H^{-1})}{\tr(\Gamma^{1}H^{-1})}\times\frac{1+K^\s C_{Y|X}}{1+ C_{Y|X}}. $$
Since $\Gamma^{K}=A+B/K$ and $K^\s=\nu(\xi)$, we get~\eqref{def_rs} after simple calculations. We have $r^\s\le 1$ since $\nu(\xi)+\frac \xi{\nu(\xi)}\le 1+\xi$, 
$r^\s\ge \frac{C_{Y|X}}{1+C_{Y|X}}$  since $\nu(\xi)\ge 1$ and $\lim_{\xi \to \infty}r^\s = \frac{C_{Y|X}}{1+C_{Y|X}}$ since $\nu(\xi)\sim_{\xi \to \infty}\sqrt{\xi}$. 
\end{proof}
\begin{remark}\label{rk_rK}
  By the same reasoning as in the proof of Corollary~\ref{cor:gain}, we can define
  \begin{equation}\label{def_rK}
    r^K=\frac{\tr(\Gamma^{K}H^{-1})}{\tr(\Gamma^{1}H^{-1})}\times\frac{1+K C_{Y|X}}{1+ C_{Y|X}}
  \end{equation}
  as the multiplicative gain resulting from using $\t^K_N$ instead of $\t^1_N$. Note that this is indeed a gain if $r^K<1$ and that we have $r^\s=r^{K^\s}$. 
\end{remark}

\subsection{Estimation of the matrices \texorpdfstring{$A$}{A} and \texorpdfstring{$B$}{B}}
\label{sec:hat_A_B}
In practice, to calculate the value of $K^\s$ given by Theorem~\ref{thm:optimK}, we need to estimate the
matrices $A$ and $B$. Let   $(X_i, Y_i^{(1)},\dots, Y_i^{(K)})_i$ be iid samples
such that for all $i$, $X_i \sim X$ and
$Y_i^{(k)} \sim \cl(Y | X = X_i)$ for $k=1,\dots,K$ being sampled independently given~$X_i$. From these samples, we
can compute $\theta_N^K$ and define
\begin{equation}
  \label{eq:Gamma_hat}
  \hat \Gamma_N^{K} = \frac{1}{N} \sum_{i=1}^N \left( \varphi(\t^{K}_N, X_i) - \inv{K} \sum_{k=1}^K f(Y_i^{(k)}) \right)^2 \nabla \varphi(\t_N^{K}, X_i)  \nabla \varphi(\t_N^{K}, X_i)^T.
\end{equation}

\begin{proposition}
  \label{prop:GammaKN}
  Assume~\ref{hyp:ui},~\ref{hyp:continuous},~\ref{hyp:unique},~\ref{hyp:C2-ui} and
  \begin{align}\label{integrab}
      \E\left[ \sup_{\t \in  C} \abs{\varphi(\t, X) \nabla \varphi(\t, X)}^2\right] < \infty, \quad  \E\left[f(Y)^2 \sup_{\t \in C} \abs{\nabla \varphi(\t, X)}^2 \right] < \infty.
    \end{align}
Then, we have $\hat \Gamma_N^{K} \underset{N\to \infty} \to \Gamma^{K}$ almost surely.
\end{proposition}
\begin{proof}
  We define the function $g_K: \R^q \times \R^d \times (\R^p)^K \to \R^{q \times q}$ by
  \begin{equation*}
    g(\t, x, (y^{(k)})_{1 \le k \le K})= \left( \varphi(\t, x) - \inv{K} \sum_{k=1}^K f(y^{(k)}) \right)^2 \nabla \varphi(\t, x)  \nabla \varphi(\t, x)^T.
  \end{equation*}
  Assumptions~\ref{hyp:ui},~\ref{hyp:continuous},~\ref{hyp:C2-ui} ensure that the function $\t \mapsto g_K(\t, X, (Y^{(k)})_{1 \le k \le K})$ is a.s. continuous, while Assumption~\eqref{integrab} gives the integrability of $\sup_{\t \in C}|g_K(\t, X, (Y^{(k)})_{1 \le k \le K})|$.  From Lemma~\ref{lem:ulln}, we get that $$ \sup_{\t \in C}\left|\frac{1}{N} \sum_{i=1}^N g_K(\t, X_i, (Y_i^{(k)})_{1 \le k \le K})-\E[g_K(\t, X, (Y^{(k)})_{1 \le k \le K})]\right| \to 0, \ a.s.$$  From Proposition~\ref{prop:slln}, $\t_N^K \to \t^\s$ a.s.  Hence, we deduce that $\hat{\Gamma}_N^{K} \to \Gamma^{K}$ a.s.
\end{proof}

\paragraph{Estimators for $A$ and $B$}
From Proposition~\ref{prop:GammaKN} and Equation~\eqref{Gamma_K_expr}, we deduce estimators of $A$ and $B$. For
$K_1,K_2\in \N^*$ such that $K_1 < K_2$,  we have by Proposition~\ref{prop:GammaKN} when $N$ tends to $+\infty$
\[
  \frac{K_2\hat \Gamma_N^{K_2} - K_1\hat \Gamma_N^{K_1}}{K_2-K_1} \to A \; a.s., \
  \frac{K_1K_2\left(\hat \Gamma_N^{K_1} - \hat \Gamma_N^{K_2}\right)}{K_2-K_1}\to B \; a.s.
\]
We will mainly use $K_1=\bar{K}$ and $K_2=2\bar{K}$ for a given $\bar{K}\in \N^*$,  which leads to simpler formulas
\[
  \hat A_{\bar{K}} := 2\hat \Gamma_N^{2\bar{K}} - \hat \Gamma_N^{\bar{K}},   \hat B_{\bar{K}} = 2\bar{K}\left(\hat \Gamma_N^{\bar{K}} - \hat \Gamma_N^{2\bar{K}}\right).
  \]
  Besides, we rather work with the following antithetic estimators:
  \begin{align*}
     \hat A^{anti}_{\bar{K}} &=   \frac{1}{N} \sum_{i=1}^N \Bigg[ 2\left( \varphi(\t^{2\bar{K}}_N, X_i) - \inv{2 \bar{K}} \sum_{k=1}^{2\bar{K}} f(Y_i^{(k)}) \right)^2- \frac 12 \left( \varphi(\t^{2\bar{K}}_N, X_i) - \inv{\bar{K}} \sum_{k=1}^{\bar{K}} f(Y_i^{(k)}) \right)^2\\ & - \frac 12\left( \varphi(\t^{2\bar{K}}_N, X_i) - \inv{\bar{K}} \sum_{k=\bar{K}+1}^{2\bar{K}} f(Y_i^{(k)}) \right)^2   \Bigg]\nabla \varphi(\t_N^{2\bar{K}}, X_i)  \nabla \varphi(\t_N^{2\bar{K}}, X_i)^T  \\
     \hat B^{anti}_{\bar{K}} &= 2\bar{K}\Bigg(  \frac{1}{N} \sum_{i=1}^N \Bigg[ \frac 12 \left( \varphi(\t^{2\bar{K}}_N, X_i) - \inv{\bar{K}} \sum_{k=1}^{\bar{K}} f(Y_i^{(k)}) \right)^2 + \frac 12\left( \varphi(\t^{2\bar{K}}_N, X_i) - \inv{\bar{K}} \sum_{k=\bar{K}+1}^{2\bar{K}} f(Y_i^{(k)}) \right)^2 \\
      &-  \left( \varphi(\t^{2\bar{K}}_N, X_i) - \inv{2\bar{K}} \sum_{k=1}^{2\bar{K}} f(Y_i^{(k)}) \right)^2\Bigg]\nabla \varphi(\t_N^{2\bar{K}}, X_i)  \nabla \varphi(\t_N^{2\bar{K}}, X_i)^T \Bigg).
  \end{align*}
  Note that the same value of $\theta_N^{2\bar{K}}$ is used. Similarly, we have the almost sure convergence of these estimators respectively to $A$ and $B$ as $N\to \infty$. Thanks to the convexity of the square function, $ \hat B^{anti}_{\bar{K}}$ is a semi-definite positive matrix. Unfortunately, the matrix $ \hat A^{anti}_{\bar{K}}$ may not be semi-definite positive. More generally, the matrix $A$ is in general difficult to estimate. From its definition~\eqref{def_A}, we see that the better is the approximation family $\varphi(\theta,X)$, the smaller is the matrix $A$ for the natural order (L\"owner order). Thus, when the conditional expectation is well approximated, the matrix~$A$ is small and may be smaller than the noise in $O(N^{-1/2})$, so that the estimated matrix $\hat A^{anti}_{\bar{K}}$ may have negative eigenvalues. Thus, in practice, we use \begin{equation}\label{def_KAH}
    \hat{K}^A_H=\nu\left(\frac{\tr( \hat B^{anti}_{\bar{K}} \hat{H}^{-1})}{C_{Y|X} \tr( (\hat A^{anti}_{\bar{K}} \hat{H}^{-1})_+)} \right)\end{equation} to approximate $K^\s$. An alternative is to approximate $A$ by $\Gamma^{2\bar{K}}_N$ for a (fixed) large value of $\bar{K}$: it is a nonnegative estimator of $\Gamma=A+B/\bar{K}\ge A$, and therefore \begin{equation}\label{def_KGH} \hat{K}^{\Gamma}_H= \nu\left(\frac{\tr( \hat B^{anti}_{\bar{K}} \hat{H}^{-1})}{C_{Y|X} \tr( \Gamma^{2\bar{K}}_N \hat{H}^{-1})} \right)
  \end{equation}
underestimates $K^\s$. These estimators are discussed and illustrated in the numerical section~\ref{Sec_num}.

\section{The linear regression framework}\label{Sec_linearcase}

In this section, we rephrase some results of Section~\ref{Sec_Main} in the framework of linear regression as they actually take simpler forms. In particular, we show that the uniform integrability assumption of Theorem~\ref{thm:optimK} is always satisfied.

\subsection{Main results for the linear regression framework}
We consider in this section a function $u:\R^d\to \R^q$ such that $\E[|u(X)|^2]<\infty$ and
$$\varphi(\t,X)=\t \cdot u(X), \ \t \in \R^q.$$

In this case, we have $\nabla \varphi(\t,X)= u(X)$, $\nabla^2 \varphi(\t,X)= 0$, $\nabla v^\infty(\t)=2\E[(\theta \cdot u(X) -\E[f(Y)|X] )u(X)]$ and $\nabla^2 v^\infty(\t)=2 \E[u(X) u^T(X)]$ does not depend on~$\theta$. Therefore, Assumptions~\ref{hyp:ui},~\ref{hyp:continuous} and~\ref{hyp:C2-ui} are clearly satisfied for any compact~$C$, while~\ref{hyp:unique} holds if, and only if
\begin{equation}\label{H3_lin}
 H=2 \E[u(X) u^T(X)] \text{ is positive definite and } \t^\s= 2H^{-1}\E[f(Y)u(X)] \in \mathring{C}.\tag{$\mathcal{H}$-3-lin}
\end{equation}
We also get a simpler expression for $\t^K_N$ and $\hat \Gamma_N^{K}$:
\begin{align*}
  \t^{K}_N&=\left(\frac 1 N  \sum_{i=1}^N u(X_i)u(X_i)^T   \right)^{-1} \left( \frac 1 N \sum_{i=1}^N \left( \frac 1 K \sum_{k=1}^K f(Y_i^{(k)})\right)u(X_i) \right), \\
  \hat \Gamma_N^{K} &= \frac{1}{N} \sum_{i=1}^N \left( \t_N^{K} \cdot u(X_i) - \inv{K} \sum_{k=1}^K f(Y_i^{(k)}) \right)^2 u(X_i)  u(X_i)^T.
\end{align*}
Here, we assume that $H$ is positive definite and $N$ is large enough so that $\frac 1 N  \sum_{i=1}^N u(X_i)u(X_i)^T $ is positive definite by the law of large numbers. However, it may be convenient to slightly modify the estimator as in the next proposition. This new estimator satisfies in particular  the uniform integrability assumption of Theorem~\ref{thm:optimK}, as shown in the proof of Proposition~\ref{prop:lin}.
\begin{definition} For a positive semi-definite matrix $S\in \R^{q\times q}$ and $\epsilon\in \R_+^*$,  $S\vee (\epsilon I_q)$ is the positive definite matrix such that   $ (S\vee (\epsilon I_q))e_l=\max(\lambda_l,\epsilon) e_l$, where $(e_l)_{1\le l\le q}$ is an orthonormal basis of eigenvectors with respective eigenvalues $(\lambda_l)_{1\le l\le q}$.
\end{definition}

\begin{proposition}\label{prop:lin}
  We assume~\eqref{H3_lin}, $\E[|u(X)|^{4+\eta}]<\infty$ and $\E[f(Y)^{2+\eta}|u(X)|^{2+\eta}]<\infty$ for some $\eta>0$. Let $\epsilon>0$ be such that $H-2\epsilon I_q$ is positive definite and  define
  $$\t^{K,\epsilon}_N=2\left(\left(\frac 2 N  \sum_{i=1}^N u(X_i)u(X_i)^T\right)\vee (\epsilon I_q) \right)^{-1} \left( \frac 1 N \sum_{i=1}^N \left( \frac 1 K \sum_{k=1}^K f(Y_i^{(k)})\right)u(X_i) \right).$$
  Then, we have $\t^{K,\epsilon}_N\to \t^\s$ a.s., $\sqrt{N} (\t_N^{K,\epsilon} - \t^\s) \xrightarrow[N \to \infty]{\cl} \cn(0, 4 H^{-1} \Gamma^{K} H^{-1})$ and
  $$\E[v^\infty(\t^K_N)] \underset{N\to \infty}{=} v^\infty(\t^\star)+\frac{\tr(\Gamma^{K}H^{-1})}N+o(1/N). $$
  The conclusions of Theorem~\ref{thm:optimK} hold.
\end{proposition}
\begin{proof}
By the law of large numbers, $\frac 2 N  \sum_{i=1}^N u(X_i)u(X_i)^T \to H$, almost surely. Since $H-2\epsilon I_q$ is positive definite, there exists, almost surely, $\bar{N}$ such that for $N\ge \bar{N}$, $\frac 2 N  \sum_{i=1}^N u(X_i)u(X_i)^T-\epsilon I_q$
is positive definite and thus $\t^{K,\epsilon}_N=\t^K_N$. This gives $\t^{K,\epsilon}_N\to \t^\s$ a.s. by Proposition~\ref{prop:slln} and $\sqrt{N} (\t_N^{K,\epsilon} - \t^\s) \xrightarrow[N \to \infty]{\cl} \cn(0, 4 H^{-1} \Gamma^{K} H^{-1})$  by Proposition~\ref{prop:tcl}.

Now, we check the uniform integrability of the sequence $N|\t^{K,\epsilon}_N-\t^\s|^2$. We have
\begin{align*}
  &\t^{K,\epsilon}_N-\t^\s=2\left[\left(\left(\frac 2 N  \sum_{i=1}^N u(X_i)u(X_i)^T\right)\vee (\epsilon I_q) \right)^{-1}-H^{-1}\right]\E[f(Y)u(X)]  \\
  &+ 2\left(\left(\frac 2 N  \sum_{i=1}^N u(X_i)u(X_i)^T\right)\vee (\epsilon I_q) \right)^{-1} \left( \frac 1 N \sum_{i=1}^N \left( \frac 1 K \sum_{k=1}^K f(Y_i^{(k)})\right)u(X_i)  -\E[f(Y)u(X)]\right)
\end{align*}
Note that for two symmetric matrices $M_1,M_2$ such that $M_1-\epsilon I_q$ and $M_2-\epsilon I_q$ are definite positive, we have $|M_1^{-1}-M_2^{-1}|=|M_1^{-1}(M_2-M_1)M_2^{-1}|\le \frac 1 {\epsilon^2} |M_2-M_1|$.  Thus, we obtain
\begin{align*}
  |\t^{K,\epsilon}_N-\t^\s|\le& \frac{2}{\epsilon^2} \left|\left(\frac 2 N  \sum_{i=1}^N u(X_i)u(X_i)^T\right)\vee (\epsilon I_q)- H \right| |\E[f(Y)u(X)] | \\&+ \frac{2}{\epsilon}\left|\frac 1 N \sum_{i=1}^N \left( \frac 1 K \sum_{k=1}^K f(Y_i^{(k)})\right)u(X_i)  -\E[f(Y)u(X)]\right|.
\end{align*}
Then, Lemma~\ref{lem_ui_TCL} with the assumptions on the moments gives the  uniform integrability of the sequence~$(N|\t^{K,\epsilon}_N-\t^\s|^2)_{N\ge 1}$.
Last, observe that $0\le v^\infty(\t)-v^\infty(\t^\s)=(\t-\t^\s)^T H(\t-\t^\s)\le |H| |\t -\t^\s|^2$. Therefore, the sequence $N(v^\infty(\t^{K,\epsilon}_N)-v^\infty(\t^\s))$ is uniformly integrable, and we get $\E[v^\infty(\t^K_N)] \underset{N\to \infty}{=} v^\infty(\t^\star)+\frac{\tr(\Gamma^{K}H^{-1})}N+o(1/N) $ as in the proof of Theorem~\ref{thm:optimK}.
\end{proof}

\subsection{Piecewise constant approximation framework}
\label{sec:local_approx}

We now specify our results in the linear case when $X$ takes its values in $[0,1]^d$, with $q=M^d$ and the basis
\begin{equation}\label{base_indic}
  u_n(x)=\prod_{j=1}^d \mathbf{1}_{I_{a_j}}(x_j),
\end{equation}
for  $n-1=a_1+a_2M+\cdots+a_dM^{d-1}$ with $a_1,\dots,a_d \in \{0,\dots,M-1\}$, and $I_a=[\frac{a}M,\frac{a+1}{M})$ for $a=0,\dots,M-2$ and $I_{M-1}=[\frac{M-1}M,1]$.

The next proposition shows that with this choice of basis, the optimal number of inner simulations is (under suitable assumptions) at least of order $M$. This illustrates that the sharper is the family $\varphi(\theta,x)$ to approximate the conditional expectation $\E[f(Y)|X]$, the larger is the optimal number of inner simulations.

\begin{proposition}\label{prop:linear_reg} Let us assume that $\E[f(Y)|X]=\psi(X)$ with $\psi:[0,1]^d\to \R$ being a Lipschitz function with Lipschitz constant $L$. Let us assume that $\E[f(Y)^2|X]-(\E[f(Y)|X])^2=\sigma^2(X)$ for a function $\sigma:[0,1]^d \to [\underline{\sigma},+\infty)$ for some $0<\underline{\sigma}<\infty$. We consider $\varphi(\theta,X)=\theta\cdot u(X)$ with $\theta\in \R^q$, $q=M^d$ with the basis defined by~\eqref{base_indic}.

Then, $\tr(AH^{-1})\le \frac 12 L^2M^{d-2}$ and $\frac 12 \underline{\sigma}^2 M^d\le  \tr(BH^{-1})$. In particular, $\frac{\tr(BH^{-1})}{\tr(AH^{-1})}\ge \underline{\sigma}^2 M^2 $ and thus $K^\star\ge \gamma M$ for some $\gamma>0$, with $K^\star$ given by Theorem~\ref{thm:optimK}.
\end{proposition}
\begin{proof}
  We have $u_n(x)=\mathbf{1}_{C_n}(x)$, with $C_n=I_{a_1}\times \cdots \times I_{a_d}$. Since $C_n\cap C_{n'}=\emptyset$ for $n\not = n'$, we have that $u(x)u(x)^T$ is a diagonal matrix. Then, the matrices $H$, $A$ and $B$ are diagonal and we get:
  \begin{align*}
    & H_{nn}=2\P(X \in C_n),\ A_{nn}=\E[(\theta^\star\cdot u(X)-\E[f(Y)|X])^2\mathbf{1}_{X\in C_n}] ,\\ &B_{nn}=\E[(f(Y)-\E[f(Y)|X])^2\mathbf{1}_{X\in C_n}]=\E[\E[f(Y)^2|X]-(\E[f(Y)|X])^2\mathbf{1}_{X\in C_n}].
  \end{align*}
  Therefore, we get $\tr(AH^{-1})=\frac 12 \sum_{n=1}^q \E[(\theta^\star\cdot u(X)-\E[f(Y)|X])^2 | {X\in C_n}]$ and $\tr(BH^{-1})=\frac 12 \sum_{n=1}^q \E[\E[f(Y)^2|X]-(\E[f(Y)|X])^2| X\in C_n]=\frac 12 \sum_{n=1}^q \E[\sigma(X)^2| X\in C_n]\ge \frac{\underline{\sigma}^2q}{2}$. Besides, we observe that $\theta^\star_n=\E[\psi(X)|X\in C_n]$ since $\theta^\star$ minimizes $v^\infty(\theta)=\sum_{n=1}^q \E[\mathbf{1}_{C_n}(X)(\theta_n-\psi(X))^2]$, and therefore
  $|\t^\star_n-\psi(x)|\le L/M$ for $x\in C_n$ by the triangular inequality. This gives $\tr(AH^{-1})\le \frac 12(L/M)^2q$, and then the claim.
\end{proof}

\begin{remark}(Asymptotic optimal tuning of $M$, $K$ and $N$) We work under the assumptions of Proposition~\ref{prop:linear_reg} and assume in addition that $\sigma(x)\le \overline{\sigma}<\infty$. We are interested in minimizing
  $$\mathcal{E}:=\E \left[  \frac 1 N \sum_{i=1}^N\left({\t}^K_N\cdot u(X_i) - \frac 1K \sum_{k=1}^Kf(Y_i^{(k)}) \right)^2\right],$$
  which is the averaged quadratic error on the sample. Following the lines of~\cite[Theorem 8.2.4]{Gobet}, we get $\mathcal{E}=\mathcal{E}_1+\mathcal{E}_2$, with
  $$\mathcal{E}_1=\E \left[  \frac 1 N \sum_{i=1}^N\left({\t}^K_N \cdot u(X_i) - \psi(X_i) \right)^2\right] \text{ and } \mathcal{E}_2=\E \left[  \frac 1 N \sum_{i=1}^N\left(\psi(X_i) - \frac 1K \sum_{k=1}^Kf(Y_i^{(k)}) \right)^2\right], $$
  representing respectively the approximation error and the statistical error. We show in the same manner that  $\mathcal{E}_1\le \frac{L^2}{M^2}$ and
  $\mathcal{E}_2 \le \frac{\overline{\sigma}^2M^d}{K N}$. To achieve a precision of order $\varepsilon>0$ (and a quadratic error of order $\varepsilon^2$), we then take   $\frac{L^2}{M^2}=\varepsilon^2$ and $\frac{\overline{\sigma}^2M^d}{K N}=\varepsilon^2$. This leads to $M=c\varepsilon^{-1}$ and $KN=c' \varepsilon^{-(2+d)}$ for some constants $c,c'>0$, for an overall computational cost of $O(\varepsilon^{-(3+d)})$. Taking the optimal choice $K^\star$ of Theorem~\ref{thm:optimK} does not change the order of convergence but improves its multiplicative rate.
\end{remark}

\section{Numerical experiments}\label{Sec_num}

The present numerical section is organized as follows. We first present the different estimators to approximate $K^\s$ and describe the Monte-Carlo algorithm that is used for all the examples. Then, we begin with a toy example that illustrates that $K^\s$ may be arbitrarily large. In this case, as the computational time $C_{Y|X}$ is equal to $C_X$, the theoretical multiplicative gain $r^\s \approx 1/2$, and we almost reach this bound in practice.   The second example shows a case where the computational time $C_{Y|X}$ is much smaller than $C_X$, and therefore where the multiplicative gain may be larger. The last example is more practically oriented and deals with risk management concerns.  

In this section, we compare the performances of the following estimations of $K^\s$ based on Theorem~\ref{thm:optimK},
\[
  K^\star_H= \nu\left(\frac{\tr(BH^{-1})}{C_{Y|X}\tr(AH^{-1})} \right)
\]
with the one suggested by Remark~\ref{rk_KnoH},
\[
  K^\star_{\noH}= \nu\left(\frac{\tr(B)}{C_{Y|X}\tr(A)} \right),
\]
which is much simpler to estimate. Following~\eqref{def_KAH} and~\eqref{def_KGH}, we approximate $K^\star_H$ or $K^\star_{\noH}$ by the four estimators
\begin{equation}
  \label{eq:Kstar-estimators}
  \begin{alignedat}{2}
    &\hat{K}^A_H = \nu\left(\frac{\tr( \hat B^{anti}_{\bar{K}} \hat{H}^{-1})}{C_{Y|X} \tr( (\hat A^{anti}_{\bar{K}} \hat{H}^{-1})_+)} \right);& \quad & 
    \hat{K}^A_{\noH} = \nu\left(\frac{\tr( \hat B^{anti}_{\bar{K}} )}{C_{Y|X} \tr( (\hat A^{anti}_{\bar{K}} )_+)} \right)\\
    &\hat{K}^\Gamma_H= \nu\left(\frac{\tr( \hat B^{anti}_{\bar{K}} \hat{H}^{-1})}{C_{Y|X} \tr( \Gamma^{2\bar{K}}_N \hat{H}^{-1})} \right);& \quad &
   \hat{K}^\Gamma_{\noH} = \nu\left(\frac{\tr( \hat B^{anti}_{\bar{K}} )}{C_{Y|X} \tr( \Gamma^{2\bar{K}}_N )} \right),
  \end{alignedat}
\end{equation}
where $\hat{H}=\nabla^2v^K_N(\t^K_N)$. Note that in the linear regression framework, we simply have $\hat{H}=\frac 1 N \sum_{i=1}^Nu(X_i)u(X_i)^T$. When $q=1$, matrices are scalar, and we take   $\hat{K}^A_H =  \hat{K}^A_{\noH} = \nu\left(\frac{ \hat B^{anti}_{\bar{K}} }{C_{Y|X} | \hat A^{anti}_{\bar{K}}| } \right)$ and $\hat{K}^\Gamma_H=\hat{K}^\Gamma_{\noH} = \nu\left(\frac{ \hat B^{anti}_{\bar{K}} }{C_{Y|X}  \Gamma^{2\bar{K}}_N } \right)$. Since $\Gamma^{2\bar{K}}\ge A$ is a semi-definite positive matrix, $\hat{K}^\Gamma_H$ and  $\hat{K}^\Gamma_{\noH}$ will slightly underestimate $ K^\star_H$ and $ K^\star_{\noH}$. However, as we will see they have a much smaller variance and give a nearly optimal computational gain.

For each example, we run our algorithm $20,000$ times to approximate $\E[v^K_N(\t^\s)] - \E[v^K_N(\t_N^K)]$, which is (under uniform integrability assumption) an estimator of $\E[v^\infty(\t_N^K)] - \E[v^\infty(\t^\s)]$ by Proposition~\ref{prop:asymp}. Namely, we calculate for $J=20,000$ the estimator $\frac 1 J \sum_{j=1}^J v^K_{N,j}(\t^\s) -  v^K_{N,j}(\t^K_{N,j})$, where $(v^K_{N,j})_{1\le j \le J}$ are iid samples of~\eqref{eq:vKN} and, for each $j$, $\t^K_{N,j}$ is the minimum of $v^K_{N,j}$. This minimum is computed explicitly for linear regression and can be approximated by a gradient descent otherwise. The value of $\t^\s$ is approximated by minimizing $v_N^1(\theta)$ for $N=100,000$. In comparison, the values of $(N,K)$ to sample $v^K_{N,j}$ and then $\t^K_{N,j}$ are such that $N \frac{1+KC_{Y|X}}{1+C_{Y|X}} \approx 5000$. This means that the simulation computational cost is fixed at 5000 times the cost of simulating~$X$, across the different values of~$K$.  

Using the $20,000$ runs, we compute as many samples of the estimators $\hat{K}^A_H$, $\hat{K}^A_{\noH}$, $\hat{K}^\Gamma_H$ and $\hat{K}^\Gamma_{\noH}$ and plot their empirical distributions on the window ${0, \dots, 110}$. To do so, we use $N=50000$ samples in the formulas~\eqref{eq:Kstar-estimators} and indicate the value of $\bar{K}$ in the captions of each related figure. Separately, we also calculate on $20,000$ runs the multiplicative computational gain $r^K$ defined by Remark~\ref{rk_rK} by using the estimator
\begin{align}\label{eq:hat_r_K}
 \hat{r}^K=\frac{\frac 1 J \sum_{j=1}^J v^K_{N'(N,K),j}(\t^\s) -  v^K_{N'(N,K),j}(\t^K_{N'(N,K),j})}{\frac 1 J \sum_{j=1}^J v^1_{N,j}(\t^\s) -  v^1_{N,j}(\t^1_{N,j})}
\end{align}
with $N=5000$ and $N'(N,K)=\left\lfloor N \frac{1+C_{Y|X}}{1+KC_{Y|X}}\right \rfloor$. In fact, assuming the uniform integrability of the family $N'(N,K)(v^K_{N'(N,K)}(\t^\s)-v^K_{N'(N,K)}(\t^K_{N'(N,K)}))$, we get by Proposition~\ref{prop:asymp} that $\E[v^K_{N'(N,K)}(\t^\s)]-\E[v^K_{N'(N,K)}(\t^K_{N'(N,K)})]\sim_{N\to \infty} \frac{\tr(\Gamma^K H^{-1})}{N'(N,K)}$ exactly as in the proof of Theorem~\ref{thm:optimK}.  By~\eqref{def_rK}, this gives
$$ \frac{\E[v^K_{N'(N,K)}(\t^\s)]-\E[v^K_{N'(N,K)}(\t^K_{N'(N,K)})]}{\E[v^1_{N}(\t^\s)]-\E[v^1_{N}(\t^1_{N})]} \to_{N\to \infty} \frac{1+KC_{Y|X}}{1+C_{Y|X}} \frac{\tr(\Gamma^K H^{-1})}{\tr(\Gamma^1 H^{-1})}=r^K.$$

\subsection{Toy example in a Gaussian framework}

Consider a one dimensional toy example in a Gaussian framework. Let $(X,Y)$ be a Gaussian vector such that $X$ and $Y$ are two standard normal random variables with covariance $\rho \in [-1, 1]$. Let $f$ be the square function, $f: x \in \R \longmapsto x^2$. We consider a constant approximation meaning that the function $\varphi$ is defined by $\varphi(\theta,x)=\theta$, for $\theta \in \R$ and $x \in \R$. Easy computations lead to explicit formulas
\begin{align*}
  &\E[f(Y)|X]= \rho^2 X^2 + (1-\rho^2) \\
  &\theta^\star=1; \quad A=\Gamma_\infty= 2 \rho^4; \quad B=2(1-\rho^4).
\end{align*}
In this case, the value of $K^\s$ is given by
\[
  K^\s=\nu\left(\frac{1-\rho^4}{\rho^4}\right).
\]
This very simple example shows that the optimal number of inner samples $K^\s$ can vary from $1$ to arbitrary large values. As the parameter $\theta$ is one dimensional, the Hessian matrix $H$ is scalar valued and therefore $ K^\s = K^\s_{\noH}$. Thus, the four estimators reduce to two.

\begin{figure}[h!tbp]
  \begin{center}
  \includegraphics[scale=0.7]{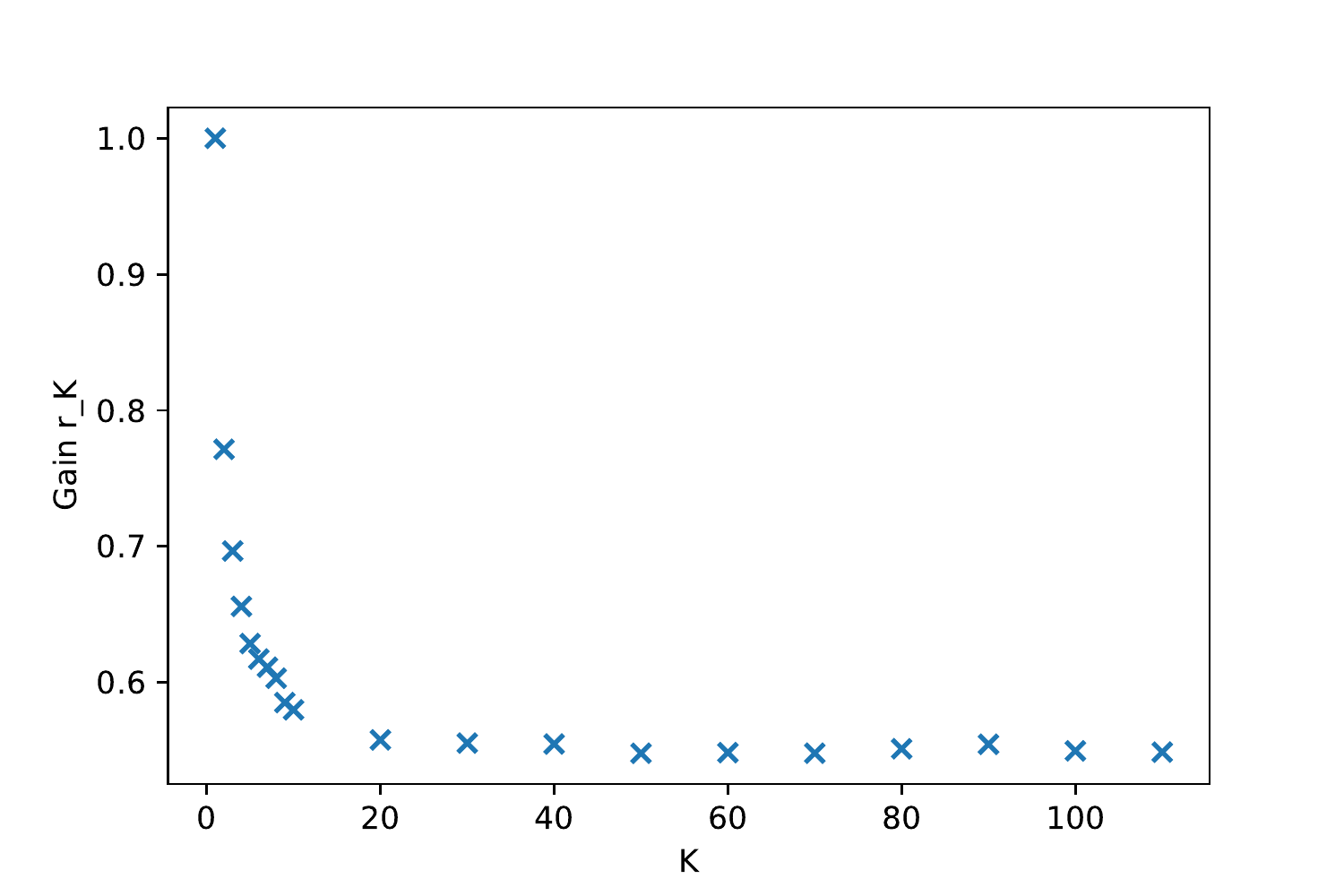}
  \end{center}
  \caption{Computational multiplicative gain as a function of~$K$ estimated with~\eqref{eq:hat_r_K} for the Gaussian toy example ($\rho = 0.1$) with regression on the constant function.}
  \label{fig:toy_efficiency}
\end{figure}

\begin{figure}[h!tbp]
  \centering\subfloat[Distribution of $\hat{K}^A_{H}$ with $\bar K=4$]{
    \label{fig:toy_K_H_A}
    \includegraphics[keepaspectratio, width=.45\textwidth]{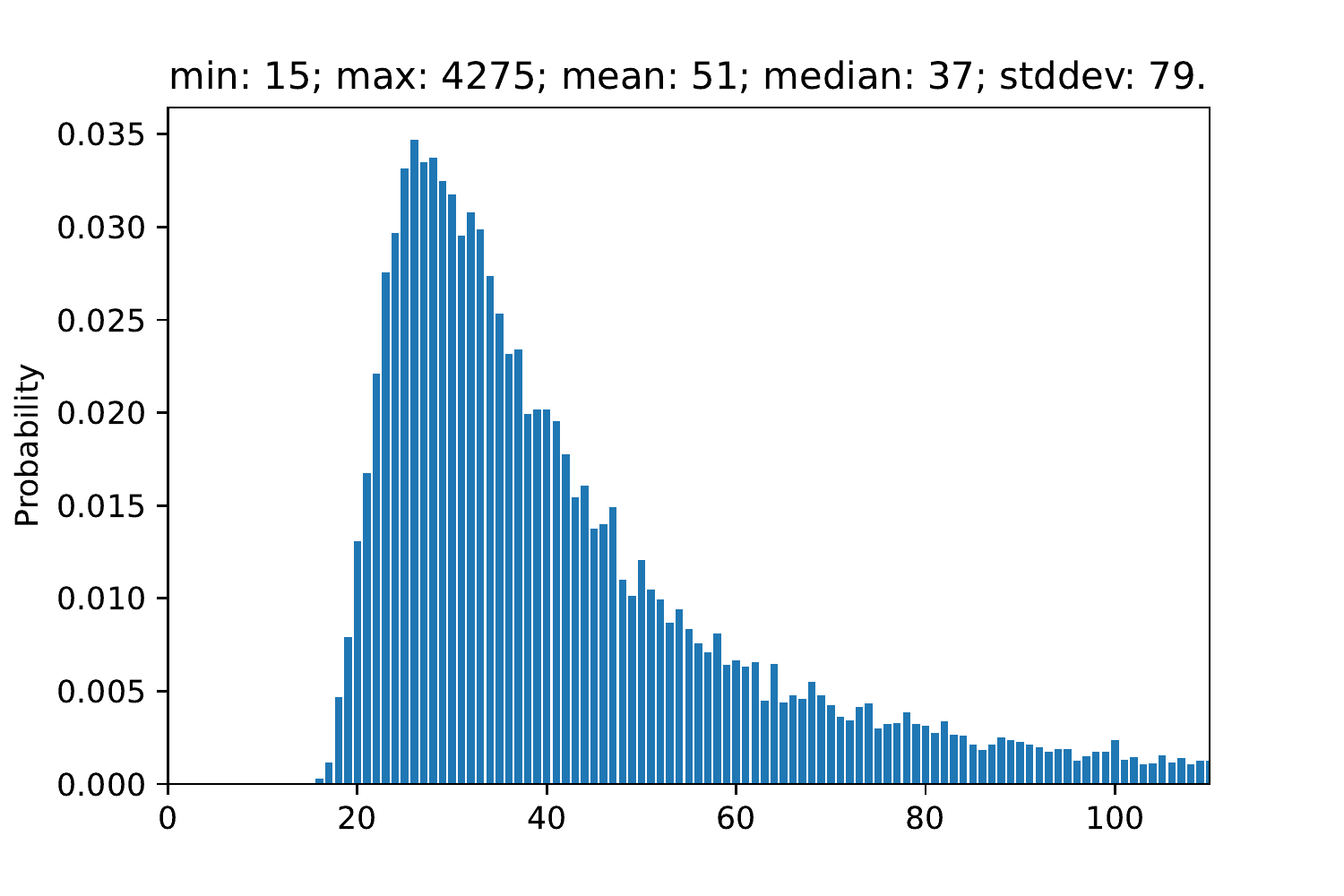}
  }
  \subfloat[Distribution of $\hat{K}^\Gamma_H$ with $\bar K=32$]{
    \label{fig:toy_K_H_Gamma}
    \includegraphics[keepaspectratio, width=.45\textwidth]{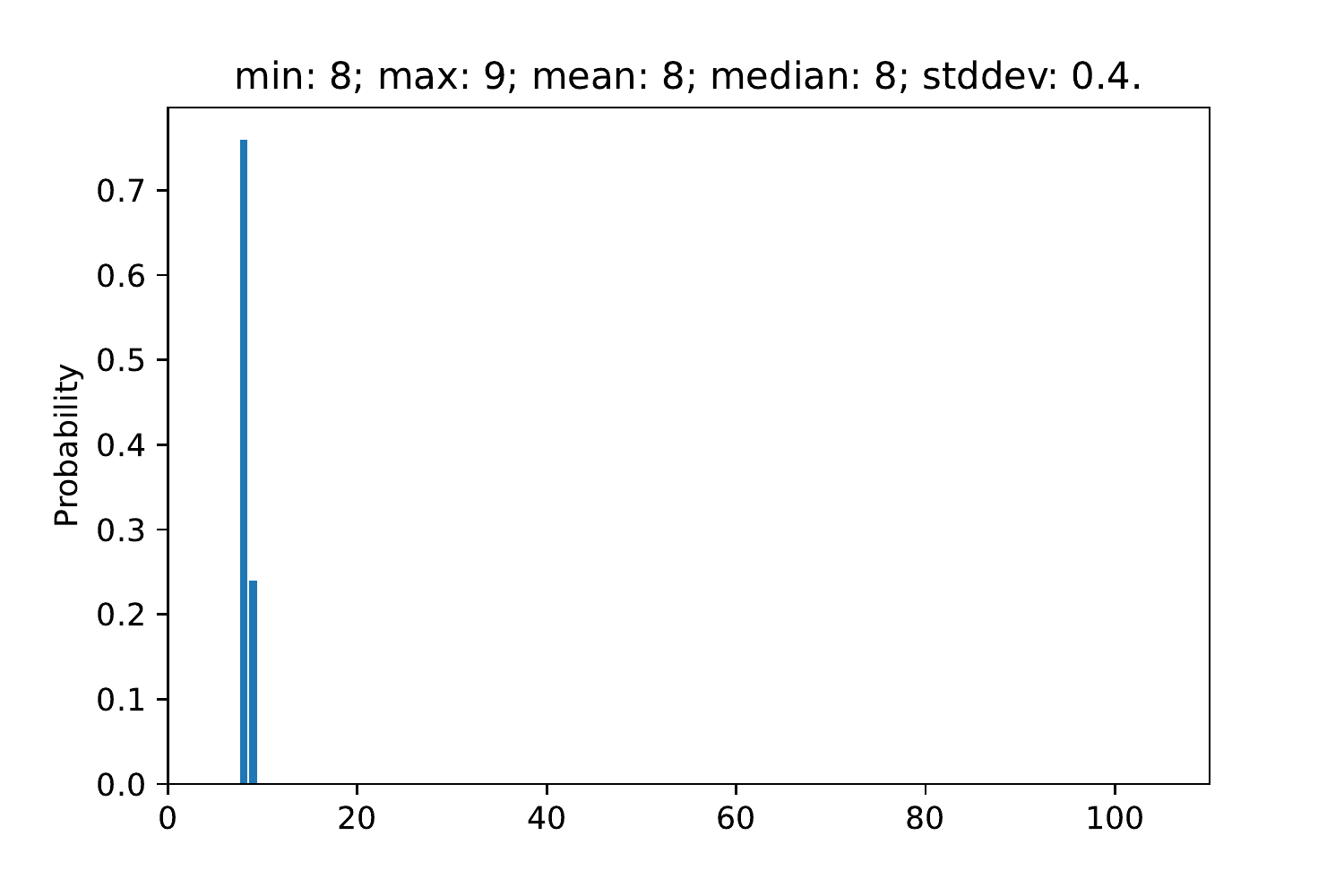}
  }
  \caption{Comparison of the 2 estimators $\hat{K}^A_H$ and $\hat{K}^\Gamma_H$ for the Gaussian toy example. }
  \label{fig:toy_K}
\end{figure}

For our numerical experiments on this toy example, we fix $\rho = 0.1$, in which case the theoretical value of $K^\s$ is $100$. Figure~\ref{fig:toy_efficiency} clearly shows that the gain $r_K$ is almost constant for any $K \ge 20$. Even though from a theoretical point of view $K^\s=100$, any values of $K$ larger than $20$ are equally good in practice. Note that from Corollary~\ref{cor:gain}, $r^\s \ge \frac{1}{2}$; this lower bound is almost attained by $r_K$ for $K \ge 20$. Figure~\ref{fig:toy_K} shows a comparison of the distributions of the two estimators $\hat{K}^A_H$, and $\hat{K}^\Gamma_H$. The estimator $\hat{K}^A_H$ has a very large standard deviation (equal to $79$) and may take values as large as $4275$, whereas the estimator $\hat{K}^\Gamma_H$ is much more concentrated and only takes two values $8$ and $9$. These are typical behaviours of these estimators: as discussed at the end of Section~\ref{sec:hat_A_B}, the estimated matrix $\hat A^{anti}_{\bar{K}}$ may have negative eigenvalues coming from a too large variance in the Monte Carlo computation and leading to non reliable estimations of $A$. On the contrary, $\hat{K}^\Gamma_H$ uses $\Gamma^{2 \bar K}_N$ as an approximation of $A$ from above leading to a conservative estimation of $K^\s$ from below with far less variance. The gains $r_K$ reported in Figure~\ref{fig:toy_efficiency} for $K=8$ or $K=9$ are very close to the best possible gains. In the example, we can conclude that $\hat{K}^\Gamma_H$ with $\bar K = 32$ is much better than $\hat{K}^A_H$, since it has small fluctuations and gives a nearly optimal computation gain. 

\subsection{A SDE conditioned on an intermediate date}\label{Subsec_SDE}

We consider the following SDE
\[
  dX_t = \cos(X_t) dW_t; \quad X_0=0
\]
where $W$ is a real valued Brownian motion. We aim at estimating $\E[X_{t_2}^2 | X_{t_1}]$ with $t_2 = 10$ and $t_1 = 9$.  This amounts to take $Y=X_{t_2}$, $f(x)=x^2$ and $X=X_{t_1}$ in~\eqref{minim_pbms}. The SDE is discretized using the Euler scheme with $200$ time-steps, hence inner simulations are cheaper than outer simulations; their relative cost is $C_{Y|X} = \frac{1}{9}$. We consider two different settings for the family of functions $(\varphi(\theta; \cdot))_\theta$: a polynomial with degree $3$ (see Figures~\ref{fig:cos_pol_efficiency_t9} and~\ref{fig:cos_sde_pol_K_t9}) and a piecewise constant approximation (see Figures~\ref{fig:cos_loc_efficiency_t9} and~\ref{fig:cos_sde_loc_K_t9}). In both settings, the parameter $\theta$ is multi-dimensional, so the Hessian matrix is a true matrix and the estimators with and without $H$ are actually different.

We build the piecewise constant approximation on $\R$ in the following way. First, we center the samples of $X$ around their mean and rescale them by their standard deviation, then we apply the  function $x \mapsto \frac{1}{\sqrt{\pi}} \int_{-\infty}^x \expp{-t^2} dt$  to map $\R$ into $(0, 1)$. Finally, we split the interval $[0, 1]$ into $M$ regular sub-intervals.

\begin{figure}[h!tbp]
  \begin{center}
  \includegraphics[scale=0.7]{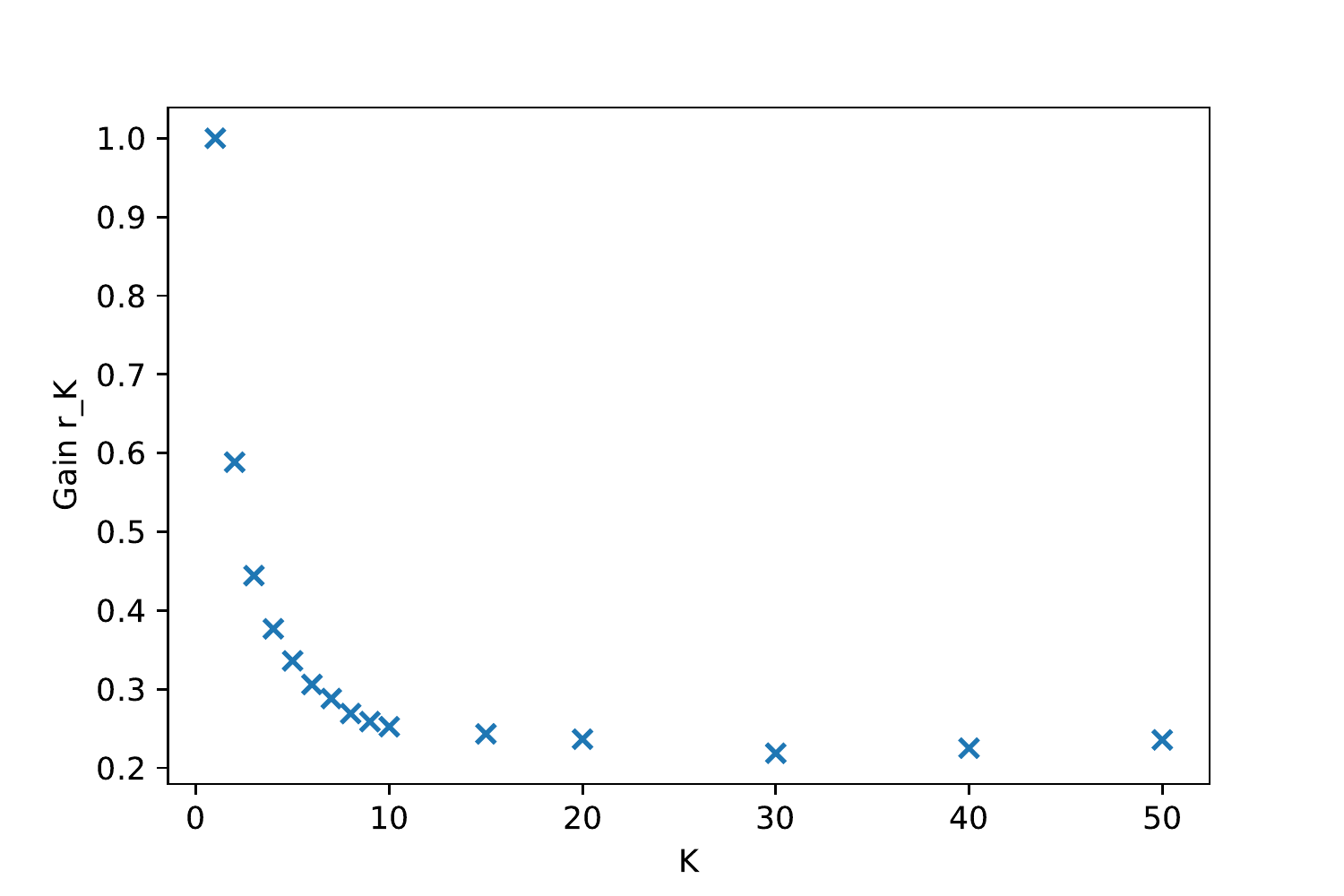}   
  \end{center}
  \caption{Computational multiplicative gain as a function of~$K$ estimated with~\eqref{eq:hat_r_K} for the SDE example with a polynomial regression of order $3$ and $t_1=9$.}
  \label{fig:cos_pol_efficiency_t9}
\end{figure}

\begin{figure}[h!tbp]   
  \begin{center}
  \includegraphics[scale=0.7]{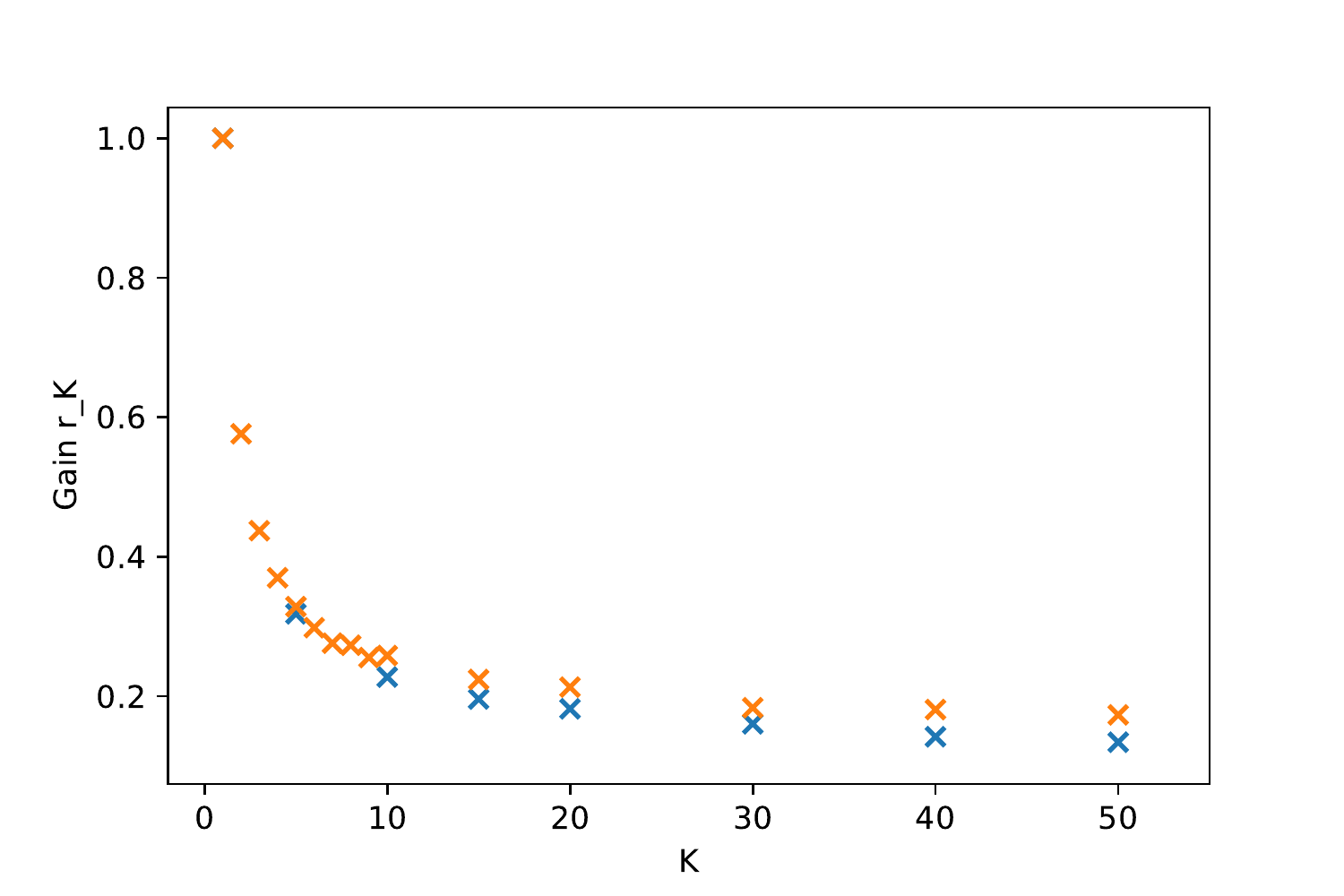}
  \end{center}
  \caption{Computational multiplicative gain as a function of~$K$ estimated with~\eqref{eq:hat_r_K} for the SDE example with a local regression with $t_1=9$ and $M=50$ (orange crosses) or $M=100$ (blue crosses).}
  \label{fig:cos_loc_efficiency_t9}
\end{figure} 

Figures~\ref{fig:cos_pol_efficiency_t9} and~\ref{fig:cos_loc_efficiency_t9} exhibit very similar gain profiles for $r^K$. The polynomial regression of order~3 works quite well on this example, which is related to the choice of $f(x)=x^2$. Figure~\ref{fig:cos_loc_efficiency_t9} compares the multiplicative gain for two different local regressions: one with $M=50$ intervals and the other one with $M=100$. As expected, increasing the number of cells improves the approximation and thus the gain, according to Corollary~\ref{cor:gain}. The multiplicative gain obtained is around $0.2$ for $M=50$ and $0.15$ for $M=100$, to be compared with the best gain $\frac{1}{10}$ given by Corollary~\ref{cor:gain}.

\begin{figure}[ht!bp]
  \centering\subfloat[Distribution of $\hat{K}^A_{H}$ with $\bar K=4$]{
    \label{fig:cos_sde_pol_K_H_A_t9}
    \includegraphics[keepaspectratio, width=.45\textwidth]{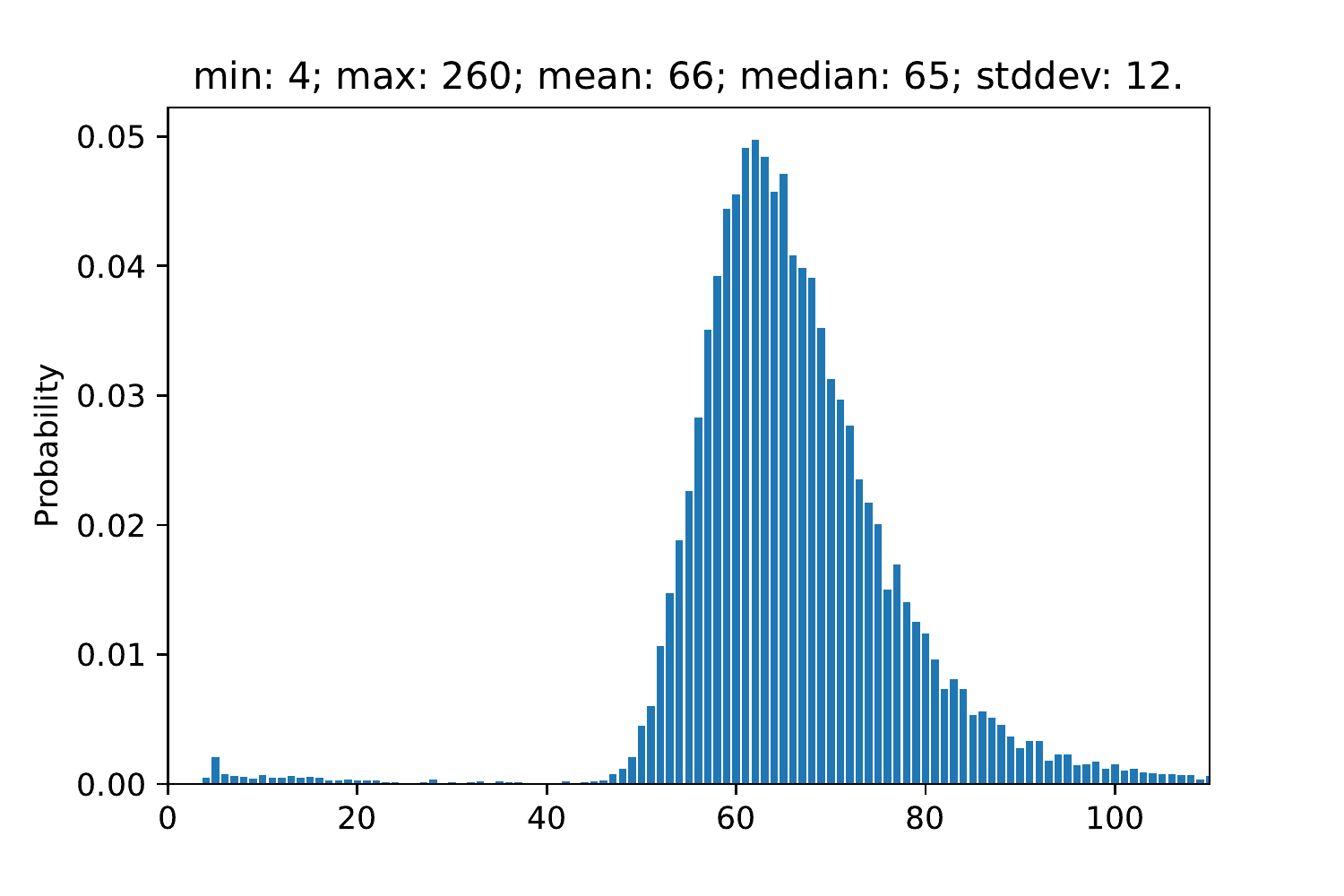}
  }
  \subfloat[Distribution of $\hat{K}^A_{\noH}$ with $\bar K=4$]{
    \label{fig:cos_sde_pol_K_noH_A_t9}
    \includegraphics[keepaspectratio, width=.45\textwidth]{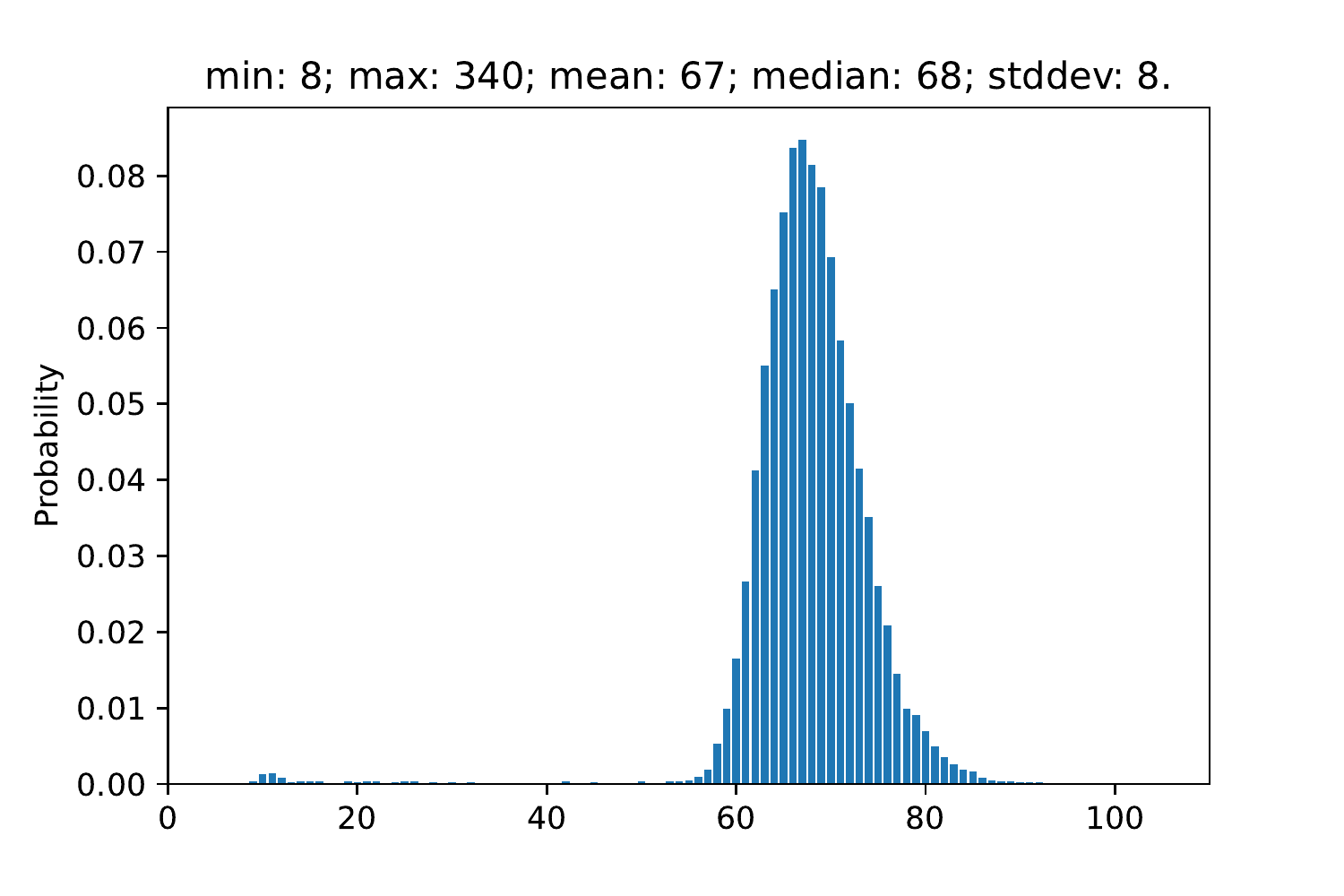}
  }\\
  \subfloat[Distribution of $\hat{K}^\Gamma_H$ with $\bar K=32$]{
    \label{fig:cos_sde_pol_K_H_Gamma2K_t9}
    \includegraphics[keepaspectratio, width=.45\textwidth]{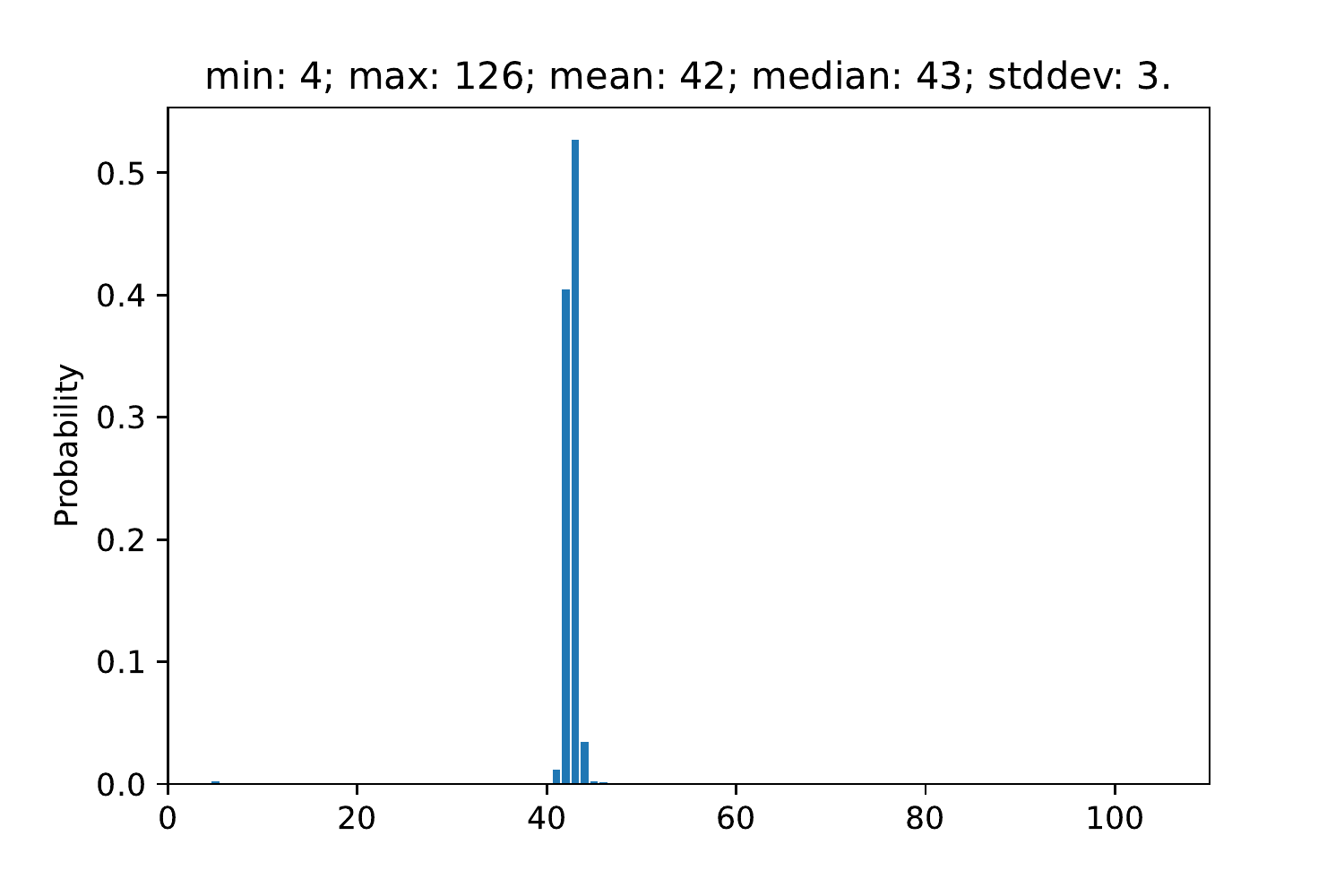}
  }
  \subfloat[Distribution of $\hat{K}^\Gamma_{\noH}$ with $\bar K=32$]{
    \label{fig:cos_sde_pol_K_noH_Gamma2K_t9}
    \includegraphics[keepaspectratio, width=.45\textwidth]{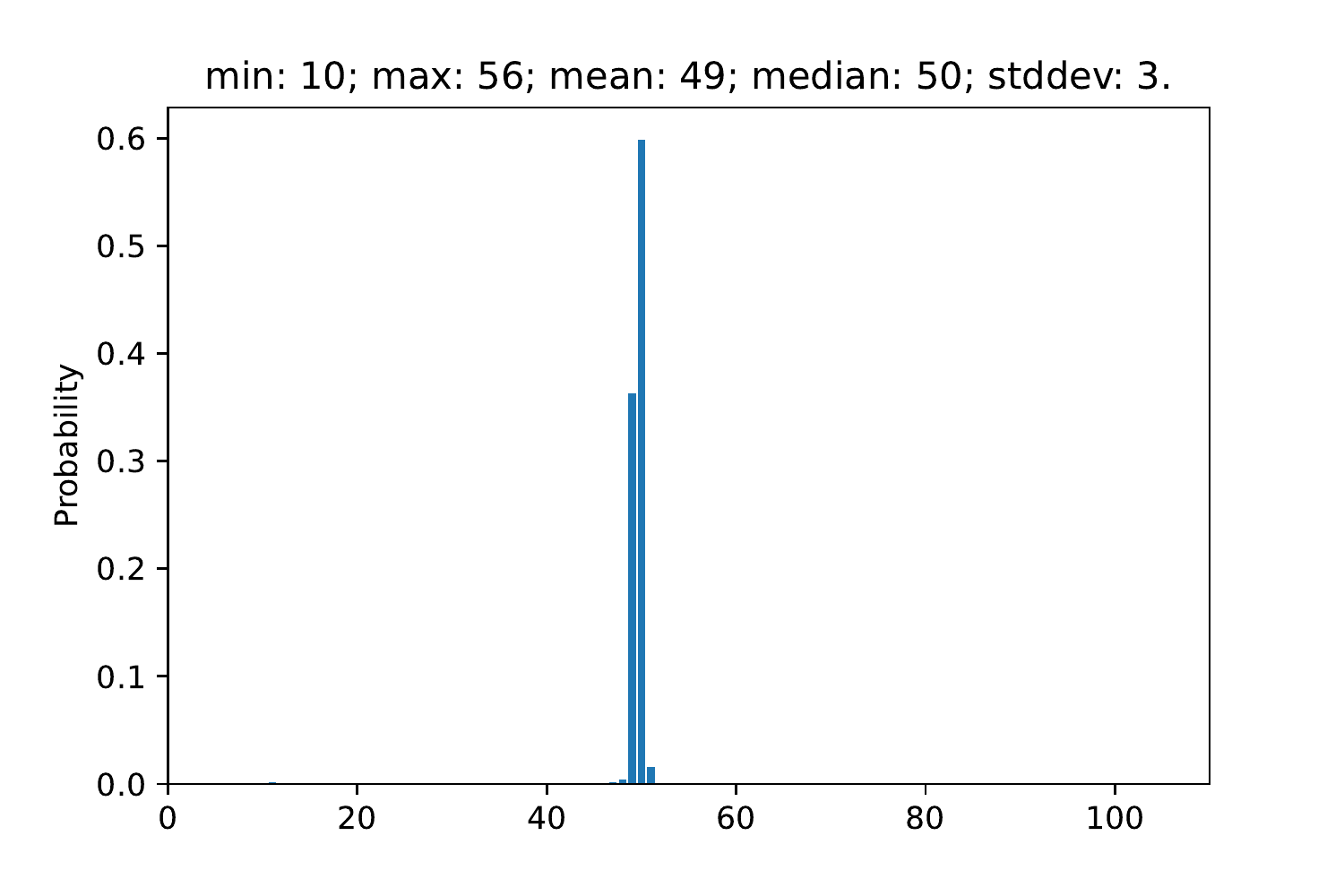}
  }
  \caption{Comparison of the 4 estimators $\hat{K}^A_H$, $\hat{K}^A_{\noH}$, $\hat{K}^\Gamma_H$ and $\hat{K}^\Gamma_{\noH}$ for the SDE example with a polynomial regression with degree $3$ and $t_1=9$.}
  \label{fig:cos_sde_pol_K_t9}
\end{figure}

 Figure~\ref{fig:cos_sde_pol_K_t9} shows a comparison of the four estimators defined in~\eqref{eq:Kstar-estimators} in the polynomial regression setting. The estimators $\hat{K}^A_H$ (resp. $\hat{K}^\Gamma_H$) and $\hat{K}^A_{\noH}$ (resp. $\hat{K}^\Gamma_{\noH}$) have very similar distributions. Simplifying $H$ in the ratio $\frac{\tr(BH^{-1})}{C_{Y|X}\tr(AH^{-1})}$ even tends to slightly reduce the variance of the estimator without significantly changing its mean. The estimators $\hat{K}^A_H$ and $\hat{K}^A_{\noH}$ based on the use of $\hat A^{anti}_{\bar{K}}$ have larger variances and may return very extreme values (between $4$ and $340$). On the contrary, the estimators $\hat{K}^\Gamma_H$ and $\hat{K}^\Gamma_{\noH}$ have very small standard deviation and show a much more concentrated probability function than the estimators based on $\hat A^{anti}_{\bar{K}}$. The use of $\Gamma_{2K}$ as an approximation of $A$ tends to produce  smaller approximations of $K^\s$: their empirical means are shifted by approximately $-20$. However, the gain profiles of Figure~\ref{fig:cos_pol_efficiency_t9} are almost flat for $K \ge 20$, hence this shift does not change the best gain attained by our method. As a conclusion, we recommend to use $\hat K_{\noH}^K$ to approximate $K^\s$.

In Figure~\ref{fig:cos_sde_loc_K_t9}, we observe very similar behaviours for the piecewise constant approximation setting as the ones we described above for the polynomial regression framework. However, we note that the estimation of the matrix $H$ is more difficult than in the previous polynomial framework, especially for the intervals with few data. This explains heuristically  why the estimators without $H$ are less noisy.

\begin{figure}[htbp]
  \centering\subfloat[Distribution of $\hat{K}^A_{H}$ with $\bar K=4$ and $t_1=9$]{
    \label{fig:cos_sde_loc_K_H_A}
    \includegraphics[keepaspectratio, width=.45\textwidth]{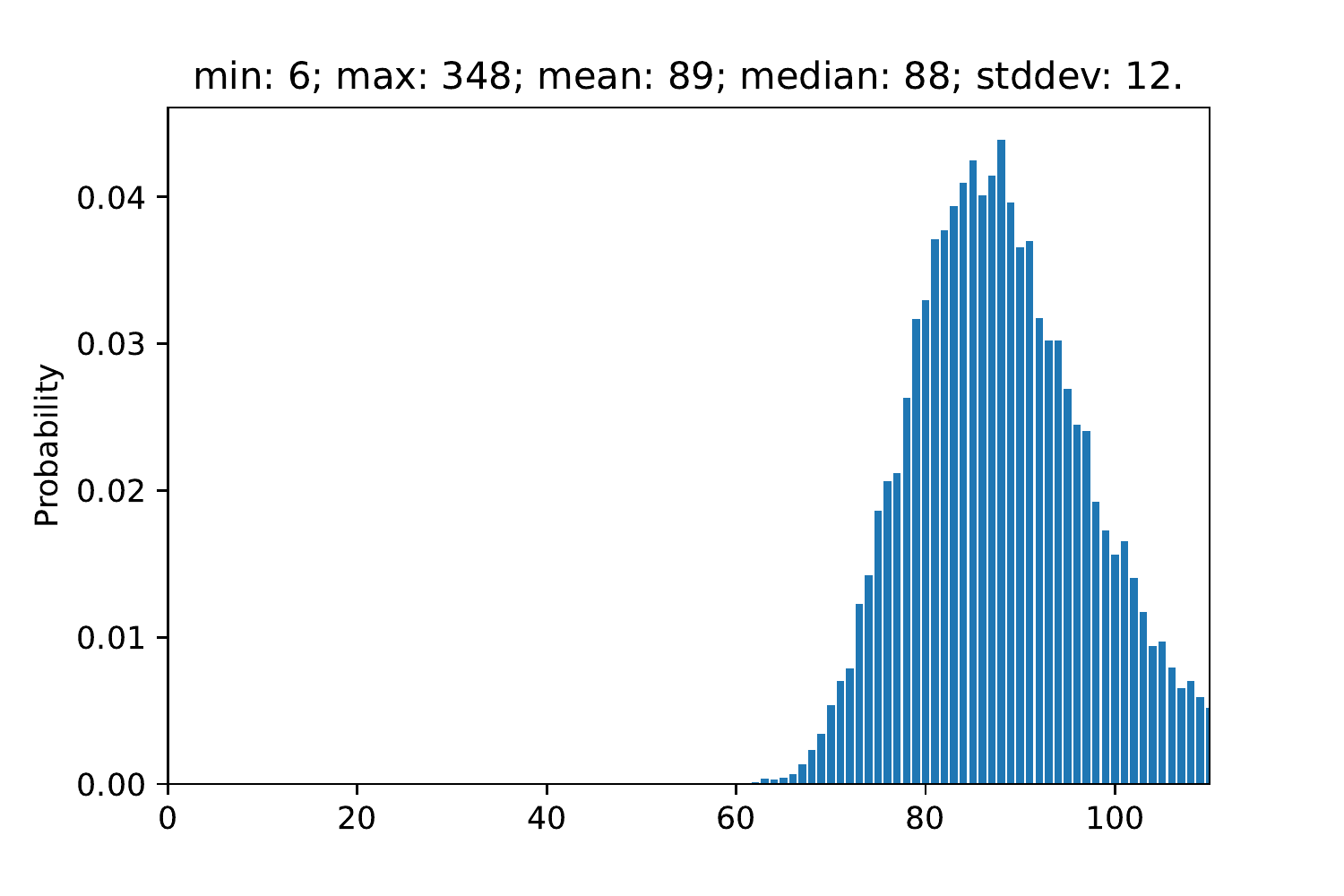}
  }
  \subfloat[Distribution of $\hat{K}^A_{\noH}$ with $\bar K=4$]{
    \label{fig:cos_sde_loc_K_noH_A}
    \includegraphics[keepaspectratio, width=.45\textwidth]{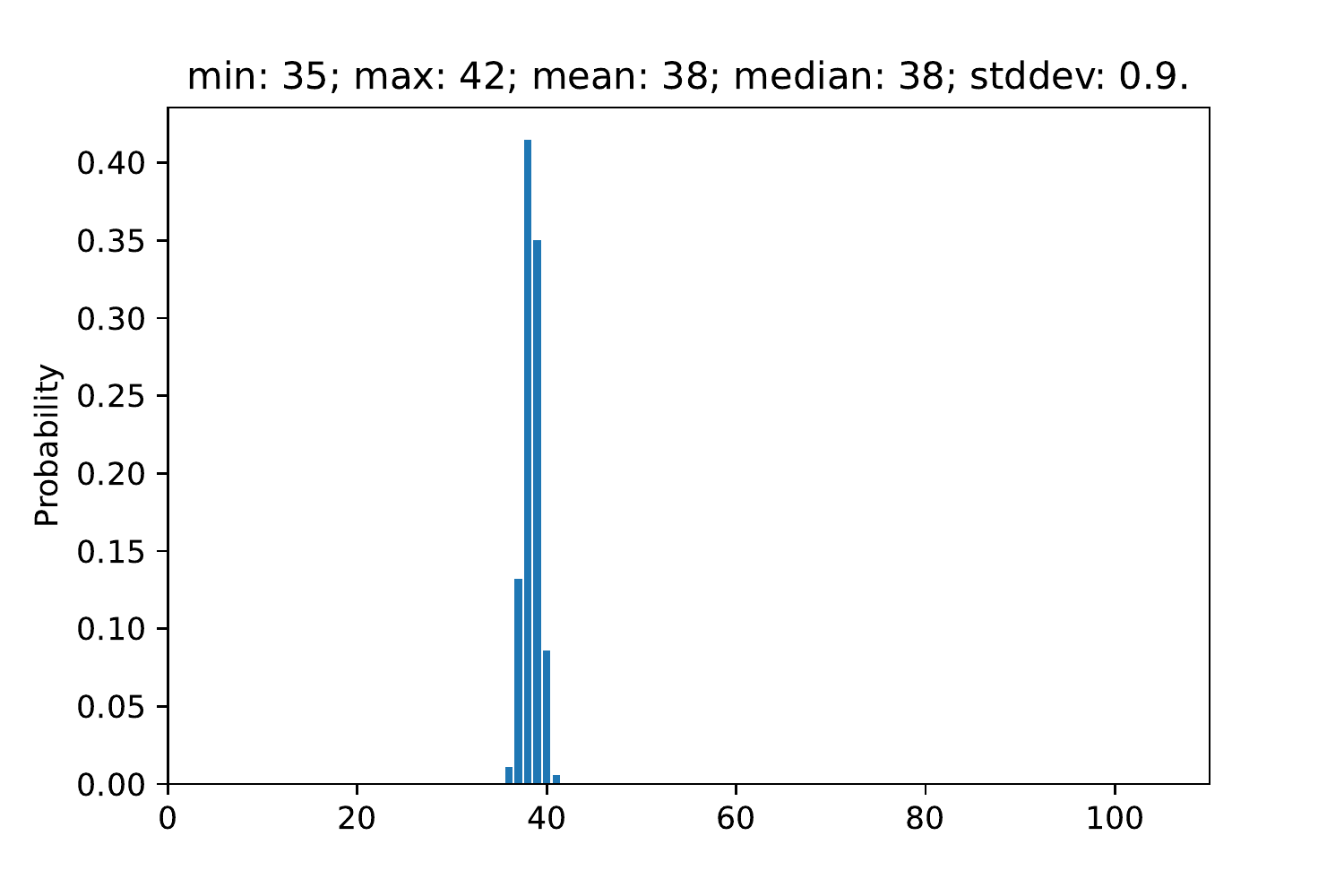}
  }\\
  \subfloat[Distribution of $\hat{K}^\Gamma_H$ with $\bar K=32$]{
    \label{fig:cos_sde_loc_K_H_Gamma2K}
    \includegraphics[keepaspectratio, width=.45\textwidth]{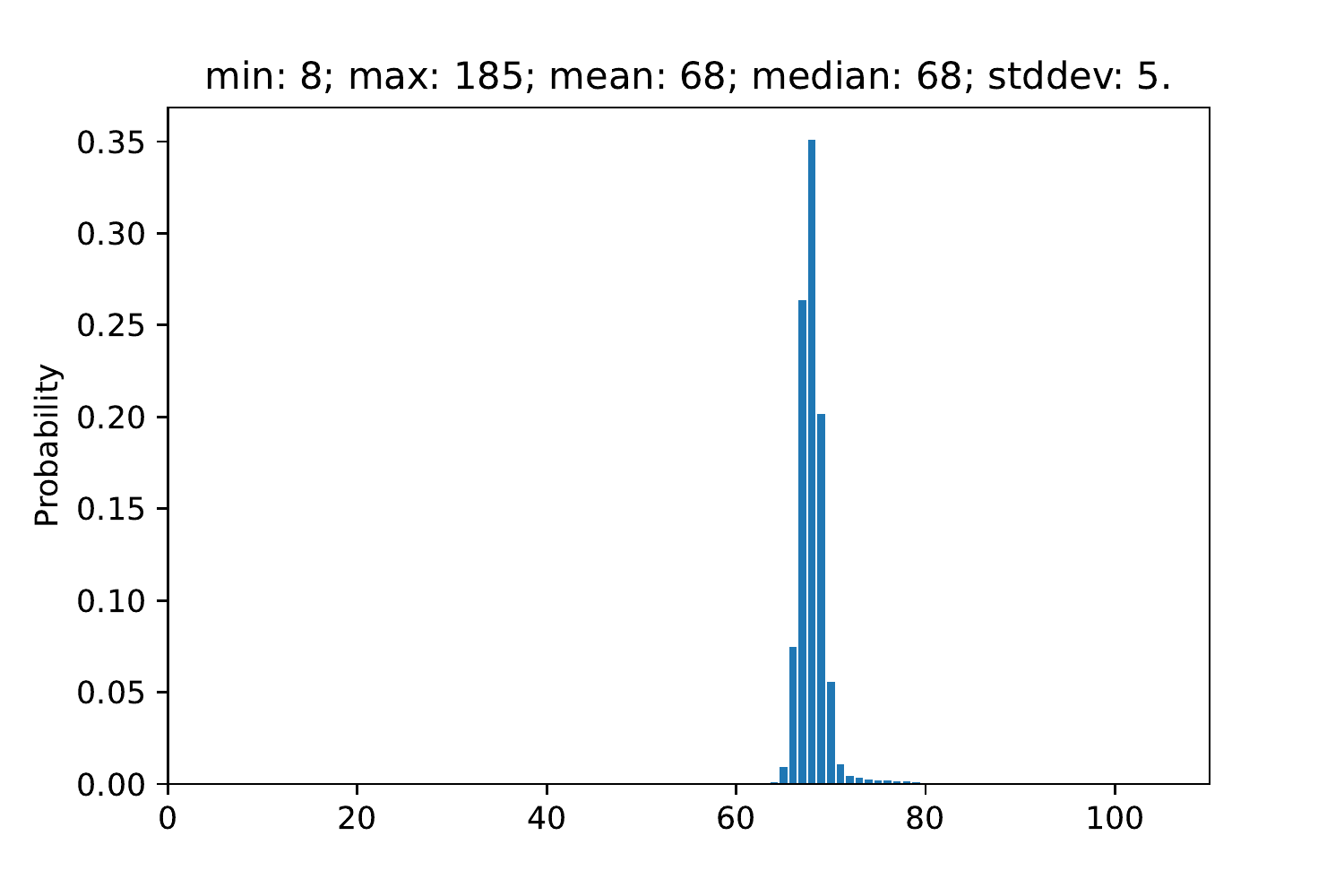}
  }
  \subfloat[Distribution of $\hat{K}^\Gamma_{\noH}$ with $\bar K=32$]{
    \label{fig:cos_sde_loc_K_noH_Gamma2K}
    \includegraphics[keepaspectratio, width=.45\textwidth]{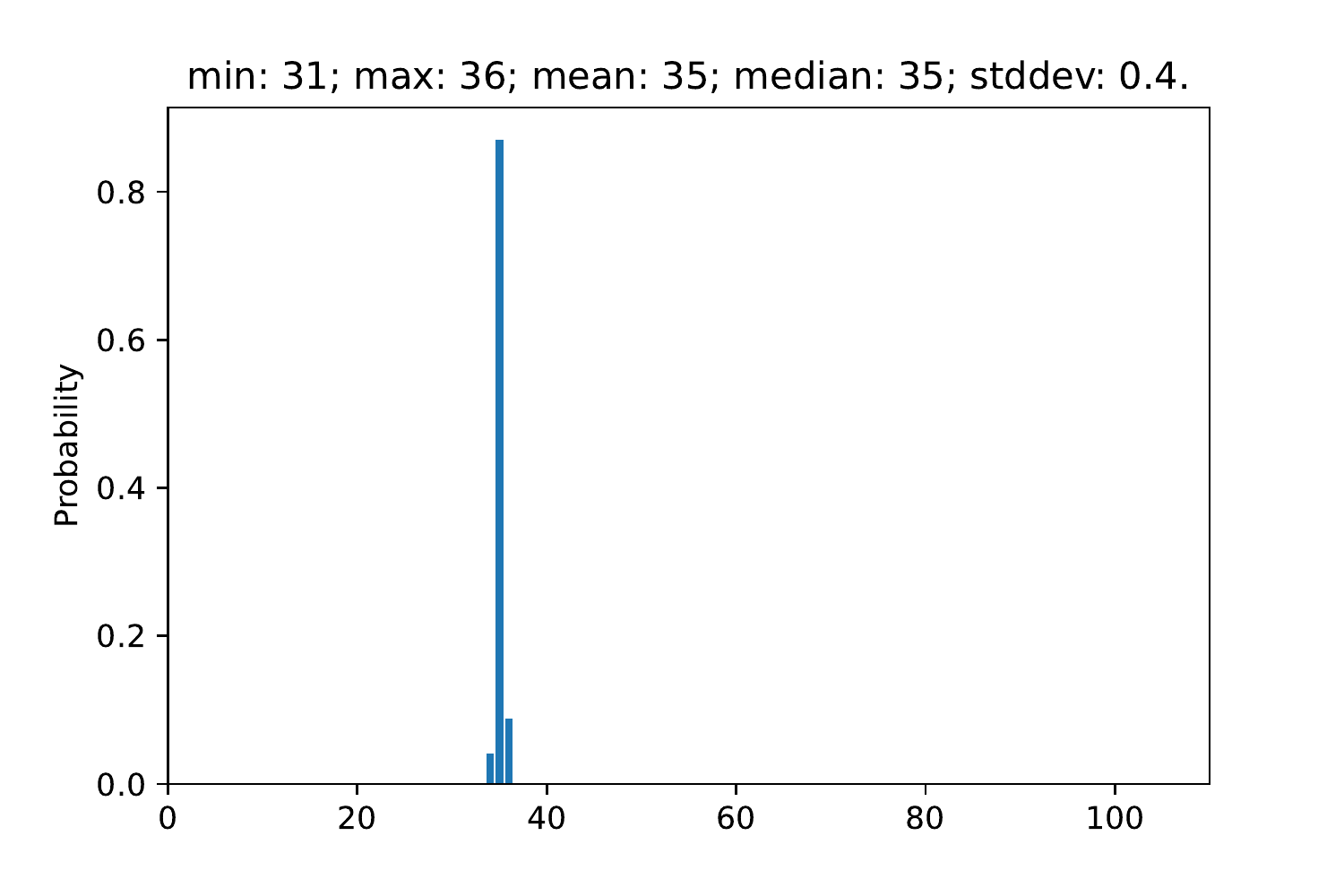}
  }
  \caption{Comparison of the 4 estimators $\hat{K}^A_H$, $\hat{K}^A_{\noH}$, $\hat{K}^\Gamma_H$ and $\hat{K}^\Gamma_{\noH}$ for the SDE example with a local regression with $M=50$ and $t_1=9$. }
  \label{fig:cos_sde_loc_K_t9}
\end{figure}

\subsection{An introductory example to risk management in insurance}

In the introduction of the present paper, we have indicated the relevance of computing conditional expectations for risk management. Here, we take back this example from~\cite{alfonsi2021multilevel} that mimics the methodology of the standard formula to calculate the Solvency Capital Requirement, in the sense that it applies a shock to the underlying asset (we refer to~\cite{alfonsi2021multilevel} for further details). We now describe this example and consider an asset whose price follows the Black-Scholes model:
$$S_t=S_0  \exp\left( \sigma W_t-\frac{\sigma^2}2t \right), \ t\ge 0,$$
where $S_0,\sigma>0$ and $W$ is a standard Brownian motion. In practice, insurance companies are interested in computing the losses of their portfolio  when a shock occurs in the economy.  Here, for simplicity, we will consider a butterfly option as a crude approximation of a true insurance portfolio. 
Thus, we are interested in a butterfly option with strikes $0<{\bf K}_1<{\bf K}_2$ that pays
$$\psi(S_T)=(S_T-{\bf K}_1)^++(S_T-{\bf K}_2)^+-2\left(S_T-\frac{{\bf K}_1+{\bf K}_2}2 \right)^+$$
at time $T>0$. The price of such an option at time~$t\in [0,T]$ is given by $\E[\psi(S_T)|S_t]$. Solvency II in its standard model assumes that there is a shock on the asset at time~$t\in (0,T)$ that multiplies its value by $1+s$, $s\in (-1,+\infty)$. Then, in the Black-Scholes model, we have to compute the following quantity
\begin{equation}
  \label{eq:loss}
 \mathcal{L}= \E[\max(\E[\psi(S_T)-\psi((1+s)S_T)|S_t],0)],
\end{equation}
which can be seen as the expected loss generated by the shock. In this particular example, $\E[\psi(S_T)-\psi((1+s)S_T)|S_t]$ has an explicit form by using the Black-Scholes formula, that we can use as a benchmark to compute the mean square error of our estimator of~\eqref{eq:loss}. 
Note that since $x\mapsto \max(x,0)$ is $1$-Lipschitz, we have
\begin{align*}
  \bigg|\E[\max(\varphi(\t,S_t),0)]&-\E[\max(\E[\psi(S_T)-\psi((1+s)S_T)|S_t],0)] \bigg|\\ &\le \sqrt{\E\left[\left( \E[\psi(S_T)-\psi((1+s)S_T)|S_t] -\varphi(\t,S_t) \right)^2 \right]}.
\end{align*}
The estimator $\t^K_N$ minimizes empirically the right hand side, which gives at the same time an upper bound on the approximation error of the expected loss.

Here, we have used our approach to compute  $\E[\psi(S_T)-\psi((1+s)S_T)|S_t]$. Thus, we have $X=S_t$, $Y=X\exp\left(\sigma(W_T-W_t)-\frac{\sigma^2}2 (T-t) \right)$ and $C_X=C_{Y|X}$ (the simulation of $X$ and of $Y$ given $X$ both require to sample one normal random variable)\footnote{Note that $C_X=C_{Y|X}$ is particular to the Black-Scholes model for which exact simulation is possible. For a more general diffusion, one typically uses a discretization scheme to approximate it, like in Subsection~\ref{Subsec_SDE}. Then, we rather get $C_{Y|X} \approx \frac{T-t}tC_X$ and the computational gain may be important when $t\to T$.}. 
We have taken $s=0.2$, $t=1$ and $T=2$ and consider the local regression with $M=50$, using the same transformation as the one used for the SDE example presented in Subsection~\ref{Subsec_SDE}.

Figure~\ref{fig:butterfly_efficiency} plots the multiplicative computational gain as a function of~$K$, while Figure~\ref{fig:butterfly_K} shows the empirical distribution of the different estimators~\eqref{eq:Kstar-estimators}. We see from Figure~\ref{fig:butterfly_K} that most of the computational gain is realized for $K\ge 5$. Similarly to the previous example, Figure~\ref{fig:butterfly_K} shows that the estimator $\hat{K}^\Gamma_{\noH}$    is  a good one to choose $K$: it has few fluctuations and avoid the issue of estimating~$\hat{H}$.

We now focus on the numerical approximation of~\eqref{eq:loss}. Figure~\ref{fig:butterfly_loss} illustrates the mean square error on the estimated expected loss as a function of~$(N,K)$ for a given computational budget, as explained in the introduction of Section~\ref{Sec_num}. More precisely, from  the sample $(\t^K_{N'(N,K),j}, 1\le j\le J)$, we compute:
$$ \frac 1 J \sum_{j=1}^J \left(\E[\max(\varphi(\t^K_{N'(N,K),j},S_t),0)|\t^K_{N'(N,K),j}]-\mathcal{L}\right)^2,$$
and plot the different values.
Here, we compute $$\E[\max(\varphi(\t^K_{N'(N,K),j},S_t),0)|\t^K_{N'(N,K),j}]=\int_{\R} \max(\varphi(\t^K_{N'(N,K),j},S_0 e^{\sigma \sqrt{t} x -\sigma^2 t/2}),0) \frac{e^{-x^2/2}}{\sqrt{2 \pi}}dx$$ and $\mathcal{L}$ by numerical integration, using the Black-Scholes formula for $\mathcal{L}$. We find $\mathcal{L}\approx 3.077$. We first note from Figure~\ref{fig:butterfly_loss} that in this example, as in all the other ones in Section~\ref{Sec_num}, the choice $K=1$ that is commonly used is suboptimal. Numerically, the optimal choice of $K$ seems to be $K^\s=8$ or $K^\s=9$, which is in line with the estimators $\hat{K}^\Gamma_{H}$ and $\hat{K}^\Gamma_{\noH}$. However, any choice of $K$ between 5 and 20 leads to an MSE that is close to the optimal one, which confirms that a precise estimation of~$K^\s$ is not needed to take the benefit of the proposed method. 

\begin{figure}[h!tbp]
  \begin{center}
  \includegraphics[scale=0.7]{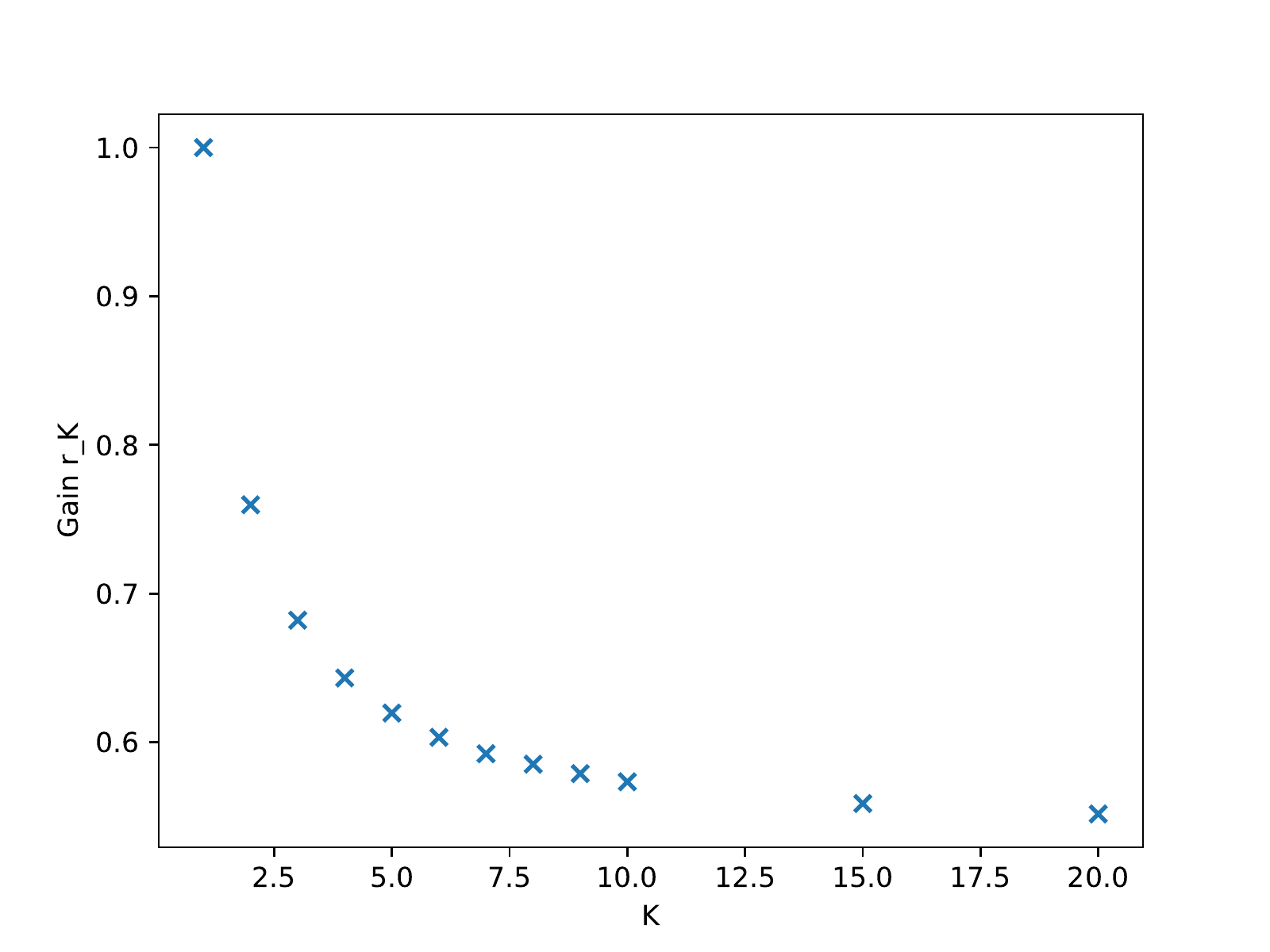}
  \end{center}
  \caption{Computational multiplicative gain as a function of~$K$ estimated with~\eqref{eq:hat_r_K} for the butterfly example with a local regression with $M=50$.}
  \label{fig:butterfly_efficiency}
\end{figure}

\begin{figure}[htbp]
  \centering
  \subfloat[Distribution of $\hat{K}^A_{H}$ with $\bar K=4$]{
    \label{fig:butterfly_K_H_A}
    \includegraphics[keepaspectratio, width=.45\textwidth]{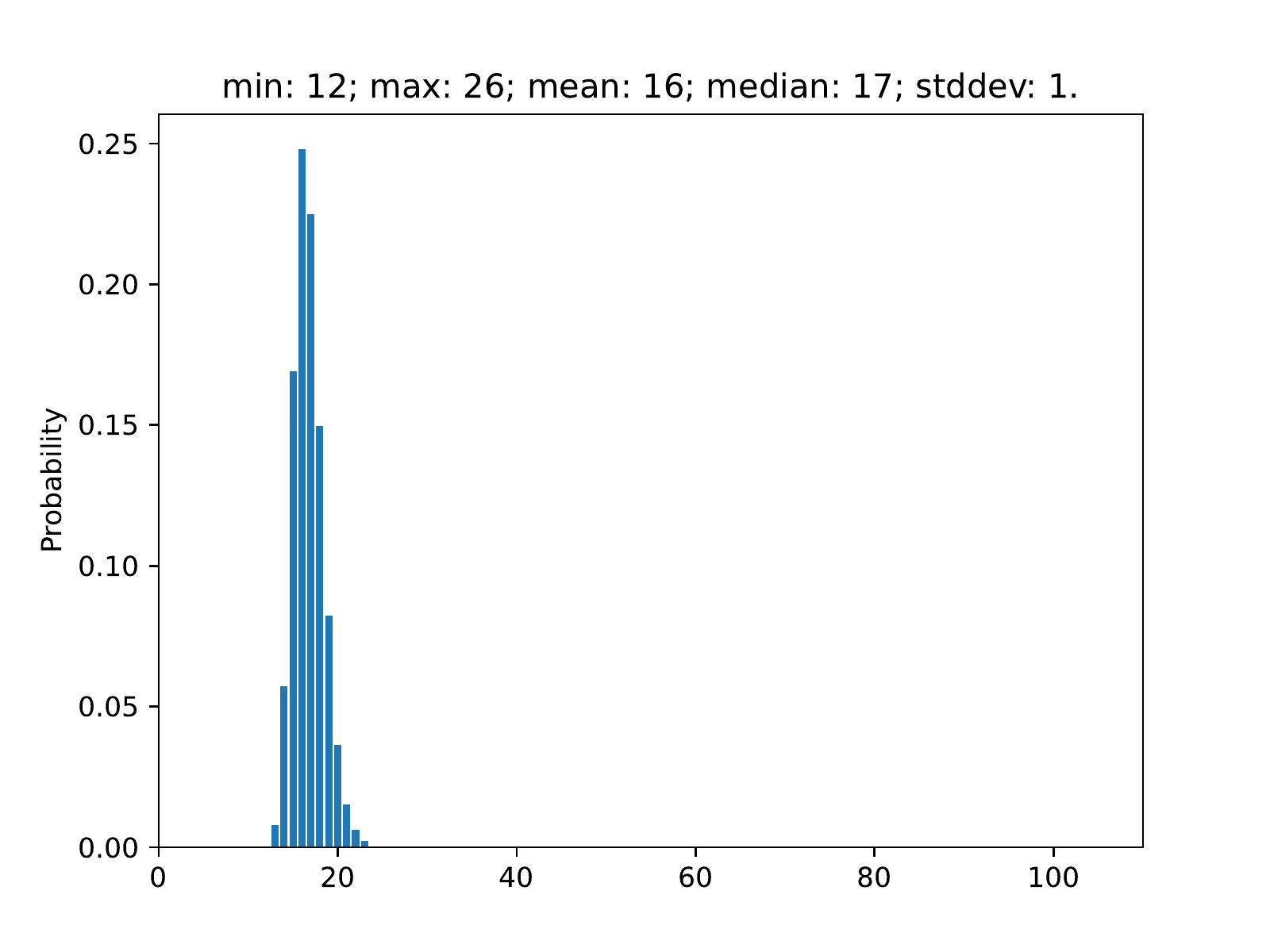}
  }
  \subfloat[Distribution of $\hat{K}^A_{\noH}$ with $\bar K=4$]{
    \label{fig:butterfly_K_noH_A}
    \includegraphics[keepaspectratio, width=.45\textwidth]{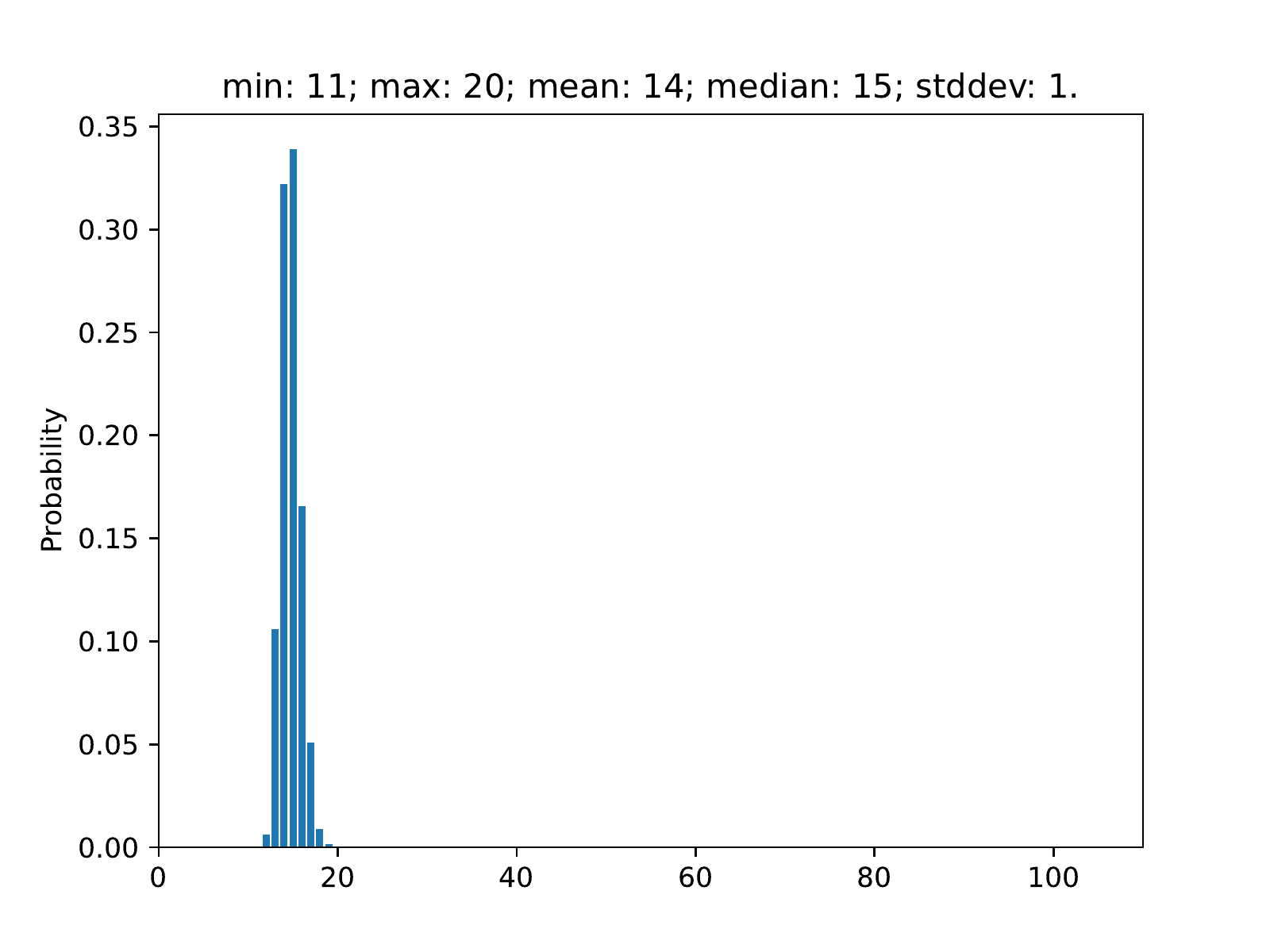}
  }\\
  \subfloat[Distribution of $\hat{K}^\Gamma_H$ with $\bar K=32$]{
    \label{fig:butterfly_K_H_Gamma2K}
    \includegraphics[keepaspectratio, width=.45\textwidth]{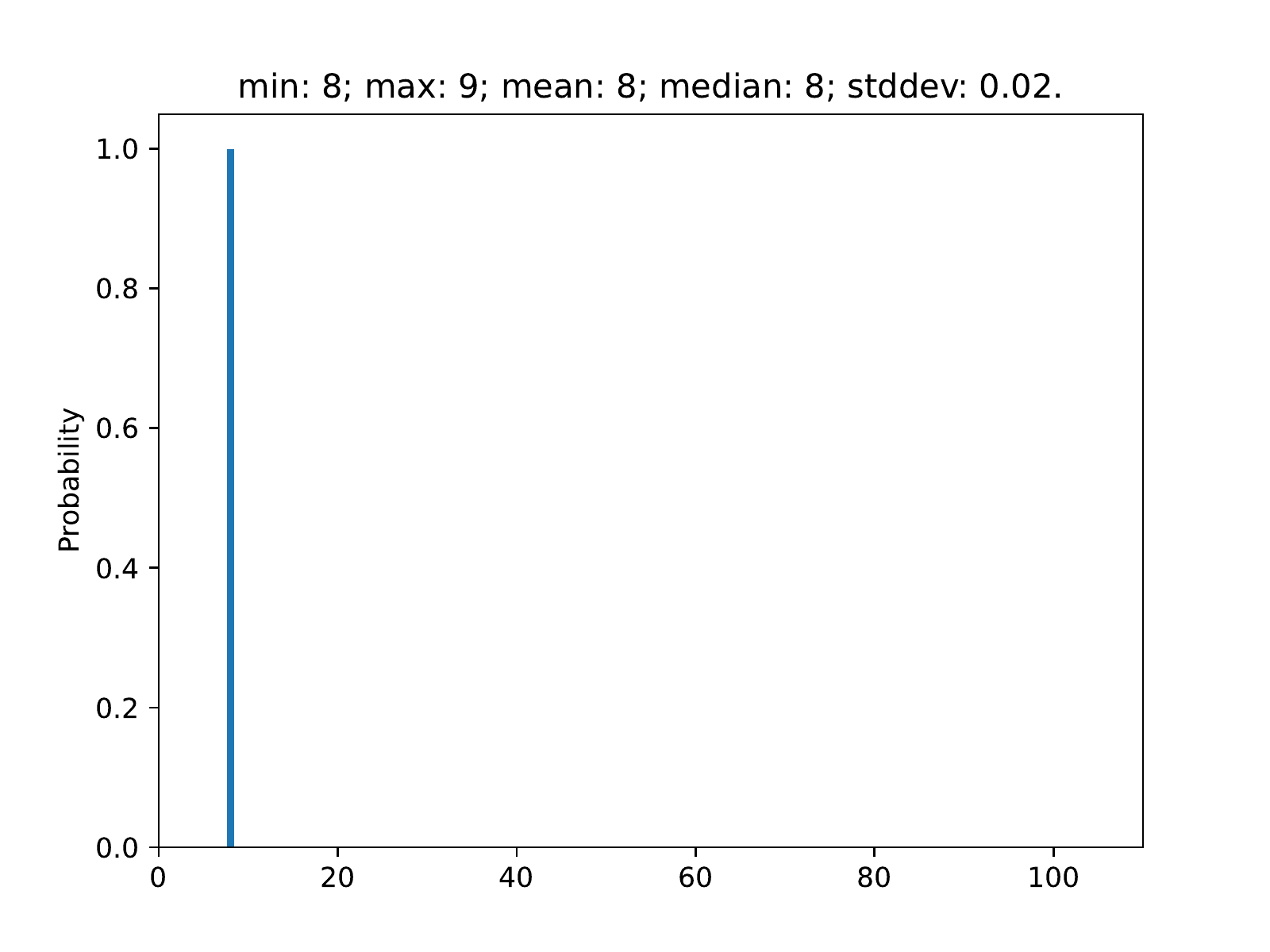}
  }
  \subfloat[Distribution of $\hat{K}^\Gamma_{\noH}$ with $\bar K=32$]{
    \label{fig:butterfly_K_noH_Gamma2K}
    \includegraphics[keepaspectratio, width=.45\textwidth]{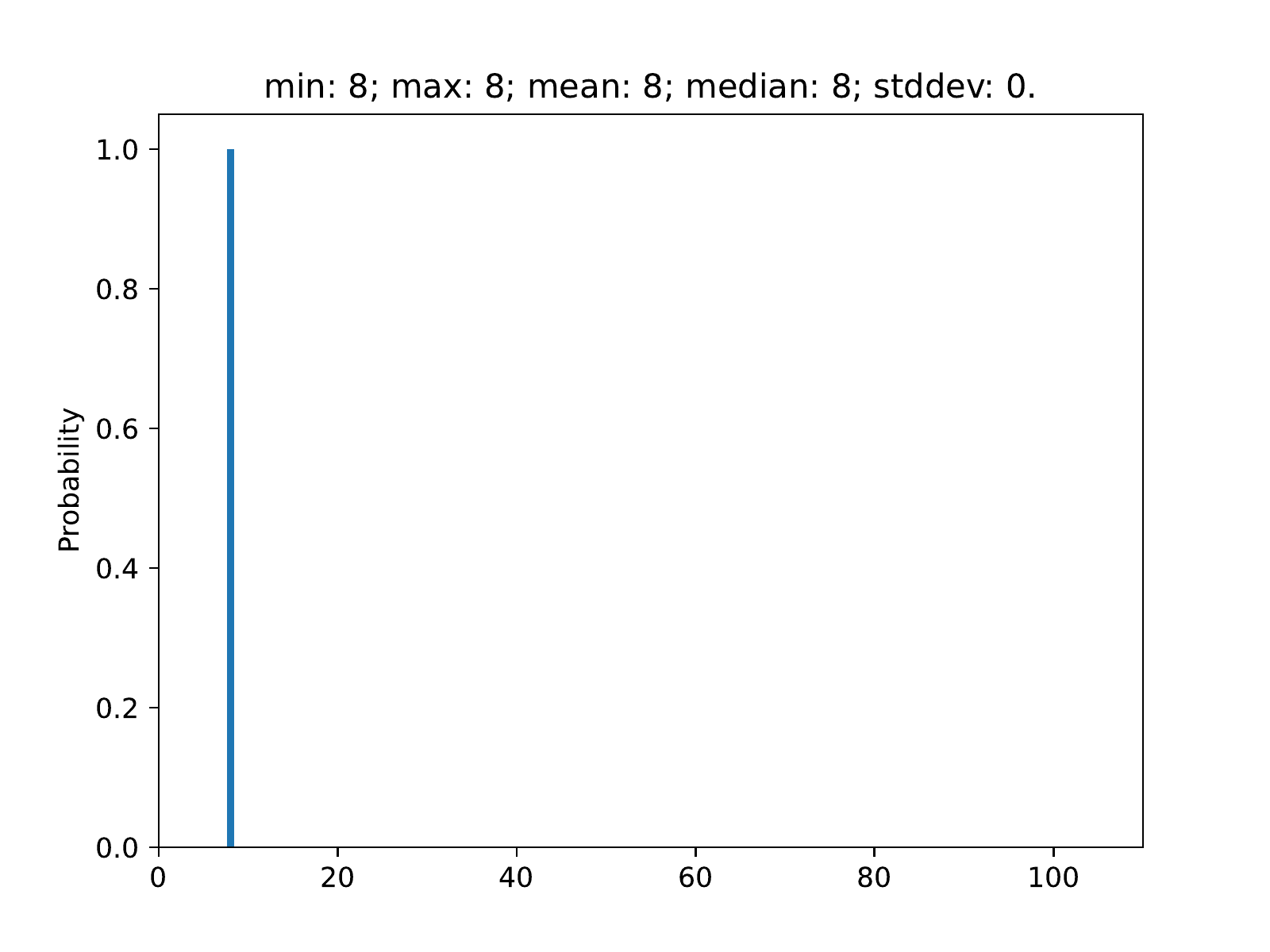}
  } 
  \caption{Comparison of the 4 estimators $\hat{K}^A_H$, $\hat{K}^A_{\noH}$, $\hat{K}^\Gamma_H$
    and $\hat{K}^\Gamma_{\noH}$ for the butterfly example with a local regression with $M=50$.}
  \label{fig:butterfly_K}
\end{figure}

\begin{figure}[h!tbp]
  \begin{center}
  \includegraphics[scale=0.7]{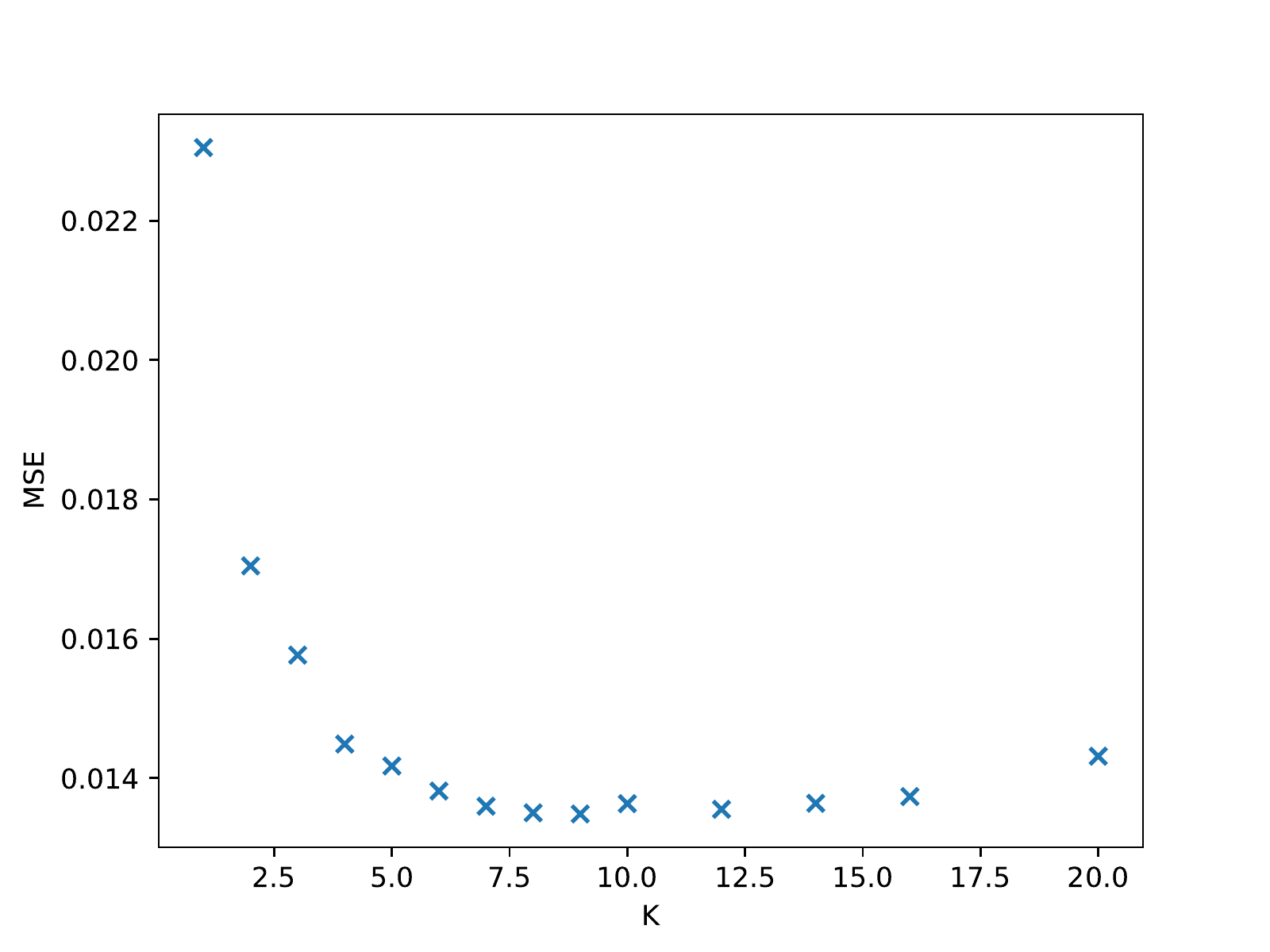}   
  \end{center} 
  \caption{Computation of the mean square error as a function of~$K$ with a local regression with $M=50$. }
  \label{fig:butterfly_loss}
\end{figure}

\newpage
\section{Conclusion}

In this work, we have investigated how to balance the computational effort between inner and outer simulations when computing conditional expectations with least-square Monte Carlo. The computational gain can be significant when the computational cost $C_{Y|X}$ is small with respect to $C_X$, and when the family $(\varphi(\t,X))_{\t}$ well approximates the conditional expectation $\E[f(Y)|X]$. 

We have proposed several estimators to approximate the optimal number of inner simulations in practice. Numerical simulations have shown that the estimators $\hat K^\Gamma_H$ and $\hat K^\Gamma_{\noH}$ have much smaller standard deviations. Although they provide smaller estimations of the optimal number of inner samples $K^\s$, they almost attain the best gain and should be used in practice in favour of those relying on $\hat A^{anti}$. When it comes to choosing between $\hat K^\Gamma_H$ and $\hat K^\Gamma_{\noH}$, one should keep in mind that $H$ is a Hessian matrix, whose computation may be extremely costly and noisy. The effect of removing $H$ in $\hat K^\Gamma$ is to reduce the noise, and  in all our experiments, this estimator  almost reaches the optimal gain. Then, as the best trade-off between accuracy and ease of computation, we suggest to use $K^\Gamma_{\noH}$.

\clearpage

\appendix
\section{Technical results}
We start by recalling  the uniform law of large number, that can be found e.g. in \cite[Lemma A1, p.~67]{ShRu93}.
\begin{lemma}\label{lem:ulln} Let $X:\Omega \to \R^d$ be a random variable, $\psi:\R^q \times \R^d\to \R$ a measurable function and $C\subset \R^q$ a compact set. Let $(X_i)_{i\ge 1}$ be an iid sequence with the same distribution as~$X$. If $C\ni \t \mapsto \psi(\t,X)$ is a.s. continuous and $\E[\sup_{\t \in C} |\psi(\t,X)|]<\infty$, then
  $$\sup_{\t\in C}\left|\frac 1N \sum_{i=1}^N \psi(\t,X_i) - \E[\psi(\t,X)] \right| \underset{N\to \infty}\to 0, \ a.s.$$
\end{lemma}
The next lemma solves the optimisation problem that arises to find the best estimator for a given computational budget.
\begin{lemma}\label{lem:optim}
  Let $a,b,\bar{c}>0$.   As $c\to +\infty$, we have
  $$\inf_{x, y\in \N^*: x+xy \bar{c} \le c}\frac{a}{x}+\frac{b}{xy}\sim_{c\to \infty} \frac 1c \left[a+b\bar{c}+ \frac{b}{a\bar{c}}\left(\nu\left( \frac{b}{a\bar{c}}\right)+\nu\left( \frac{b}{a\bar{c}}\right)^{-1} \right)\right],$$
 and the following solution is asymptotically optimal
\begin{equation}\label{optim_discret}
    x^\star(c)=\left\lfloor \frac{c}{1+y^\star \bar{c}} \right \rfloor, \ y^\star=\nu\left( \frac{b}{a\bar{c}}\right),
\end{equation}
in the sense that it satisfies $x^\star(c)+x^\star(c)y^\star \bar{c}\le c$ and
$$ \inf_{x, y\in \N^*: x+xy \bar{c} \le c}\frac{a}{x}+\frac{b}{xy} \sim_{c\to \infty} \frac{a}{x^\star(c)}+\frac{b}{x^\star(c)y^\star}.$$
\end{lemma}
\begin{proof}
  We  consider the semi-discrete minimization problem $\inf_{x>0, y\in \N^*: x+xy \bar{c} \le c}\frac{a}{x}+\frac{b}{xy}$. For each $y \in \N^*$, the optimal choice is to take $x=\frac{c}{1+y\bar{c}}$, and the infimum is given by
  $$g(y):=\frac{a+b\bar{c}}{c}+\frac{b}{cy}+\frac{a\bar{c}}{c}y.$$
  Let $y^\star_0=\sqrt{\frac{b}{a\bar{c}}}$. We check easily that $g$ is decreasing on $(0,y^\star_0)$ and increasing on $(y^\star_0,+\infty)$. Therefore the minimum on $\N^*$ is reached by~$1$ if $y^\star_0\le 1$, and by $p$ or $p+1$ if $y^\star_0\in [p,p+1]$ for some $p\in \N^*$. We compare these two candidates and  rewrite
  $g(y)=\frac{a}{c}+\frac{b\bar{c}}{c}+\frac{a\bar{c}}{c} \left(y+\frac{(y^\star_0)^2}{y}\right)$.
  Since $$p+\frac{z^2}{p} \le p+1+\frac{z^2}{p+1} \iff z^2\le p(p+1),$$ we get that $p$ is optimal if $(y^\star_0)^2\le p(p+1)$ and $p+1$ is optimal if  $(y^\star_0)^2\ge p(p+1)$ (both are optimal for $(y^\star_0)^2= p(p+1)$). Therefore, the infimum is reached by $y^\star=\nu\left( \frac{b}{a\bar{c}}\right)$ (see Definition~\ref{def_nu}), and we have
  $$\inf_{x>0, y\in \N^*: x+xy \bar{c}=c}\frac{a}{x}+\frac{b}{xy}\sim_{c\to \infty} \frac 1c \left[a+b\bar{c}+ \frac{b}{a\bar{c}}\left(\nu\left( \frac{b}{a\bar{c}}\right)+\nu\left( \frac{b}{a\bar{c}}\right)^{-1} \right)\right] .  $$
  Now, we simply notice that $\inf_{x>0, y\in \N^*: x+xy \bar{c}=c}\frac{a}{x}+\frac{b}{xy}\le \inf_{x, y\in \N^*: x+xy \bar{c}=c}\frac{a}{x}+\frac{b}{xy}$ and that
  $\frac{a}{x^\star(c)}+\frac{b}{x^\star(c)y^\star}\sim_{c\to \infty} \frac 1c \left[a+b\bar{c}+ \frac{b}{a\bar{c}}\left(\nu\left( \frac{b}{a\bar{c}}\right)+\nu\left( \frac{b}{a\bar{c}}\right)^{-1} \right)\right]$.
\end{proof}
The next lemma gives a sufficient condition to get some uniform integrability in the central limit theorem.
\begin{lemma}\label{lem_ui_TCL} Let $(Z_i)_{i\ge 1}$ be an iid sequence of random variables in $\R^d$ such that $\E[|Z_1|^{2+\eta}]<\infty$ for some $\eta>0$. Let $\bar{Z}_N=\frac 1N \sum_{i=1}^N Z_i$. Then, the sequence $(N|\bar{Z}_N-\E[Z_1]|^2)_{N\ge 1}$ is uniformly integrable.
\end{lemma}
\begin{proof}
This is a direct application of~\cite[Proposition 2.4]{GHJVW} that gives $$\E[  (N|\bar{Z}_N-\E[Z_1]|^2)^{1+\eta/2}]=\E[  (\sqrt{N}|\bar{Z}_N-\E[Z_1]|)^{2+\eta}]\le C_{2+\eta} \E[|Z_1-\E[Z_1]|^{2+\eta}],$$ for some constant $C_{2+\eta}<\infty$.
\end{proof}

\bibliographystyle{alpha}
\bibliography{biblio.bib}

\end{document}